\documentclass[oneside,12pt]{amsart}

\usepackage{amssymb}
\usepackage{amsxtra}
\usepackage{amsmath}
\usepackage{amsthm}
\usepackage[all]{xy}

\DeclareMathOperator{\GL}{GL}
\DeclareMathOperator{\PGL}{PGL}
\DeclareMathOperator{\SL}{SL}
\DeclareMathOperator{\Sp}{Sp}
\DeclareMathOperator{\Spin}{Spin}

\DeclareMathOperator{\DD}{\mathcal D}

\DeclareMathOperator{\HH}{H}
\DeclareMathOperator{\HZ}{Z}

\DeclareMathOperator{\MM}{M}
\DeclareMathOperator{\RR}{\mathcal R}
\DeclareMathOperator{\Lie}{Lie}
\DeclareMathOperator{\Cent}{Cent}

\DeclareMathOperator{\Cocent}{Cocent}
\DeclareMathOperator{\Spec}{Spec}

\DeclareMathOperator{\Gm}{{\mathbf G}_m}
\DeclareMathOperator{\Ker}{Ker\,}

\DeclareMathOperator{\Aut}{Aut\,}

\DeclareMathOperator{\Hom}{Hom\,}

\DeclareMathOperator{\Isomext}{Isomext\,}
\DeclareMathOperator{\Isomint}{Isomint\,}
\DeclareMathOperator{\Dyn}{Dyn}

\DeclareMathOperator{\Br}{Br}

\DeclareMathOperator{\Pic}{Pic}

\DeclareMathOperator{\ind}{ind}
\DeclareMathOperator{\cores}{cores}
\DeclareMathOperator{\tf}{\mathfrak{t}}
\DeclareMathOperator{\La}{\Lambda}
\DeclareMathOperator{\Lar}{\Lambda_r}
\DeclareMathOperator{\cd}{cd}

\DeclareMathOperator{\ZZ}{{\mathbb Z}}
\DeclareMathOperator{\QQ}{{\mathbb Q}}

\DeclareMathOperator{\CC}{{\mathbb C}}
\DeclareMathOperator{\UF}{\mathfrak{U}}

\theoremstyle{plain}

\newtheorem{lem}{Lemma}
\newtheorem*{thm*}{Theorem}
\newtheorem{thm}{Theorem}
\newtheorem{prop}{Proposition}

\newtheorem*{cor*}{Corollary}

\begin{document}

\title{Tits indices over semilocal rings}

\thanks{The first author is partially supported by PIMS fellowship, RFBR~08-01-00756,
RFBR~09-01-90304 and RFBR~09-01-91333. Both authors are supported by RFBR~09-01-00878.}
\author{V. Petrov}
\address{Max-Planck-Institut f\"ur Mathematik, Bonn, Germany}\email{victorapetrov@googlemail.com}
\author{A. Stavrova}\address{St. Petersburg State University, St. Petersburg, Russia\\
and Mathematisches Institut, Ludwig-Maximilians Universit\"at, M\"unchen, Germany}
\email{a\_stavrova@mail.ru}

\maketitle

\begin{abstract}
We give a simplified proof of Tits' classification of semisimple
algebraic groups that remains valid over semilocal rings.
In particular, we provide explicit necessary and sufficient
conditions that anisotropic groups of a given type appear as
anisotropic kernels of semisimple groups of a given Tits index.
We also give a new proof of the existence of all indices of exceptional inner type using the notion
of canonical dimension of projective homogeneous varieties.
\end{abstract}



\section{Introduction}

In his famous paper \cite{Tits66} Jacques Tits showed that any semisimple group $G$ over a field is
determined by its anisotropic kernel and a combinatorial datum called the \emph{Tits index} of $G$.
Some arguments were sketched or omitted there, and appeared in later papers. Namely, Selbach~\cite{Selb}
clarified the proof of the completeness of the list of Tits indices, and in~\cite{Tits90} Tits himself
has finished the proof of the existence of all indices; see also~\cite{Spr} and~\cite[Appendix]{TitsWeiss}
for more detailed expositions.

The goal of
the present paper is to show that the Tits classification carries over to arbitrary connected semilocal rings.
We do not make use of the field case, but rather provide a shortened and simplified version
of Tits' arguments. We also give a new proof of the existence of all indices of inner exceptional type
using the notion of canonical dimension of projective homogeneous varieties of semimple algebraic groups.

Our proof of Tits classification consists of two parts: combinatorial and re\-presen\-tation-theoretic.
Combinatorial restrictions
follow from the presence of the opposition involution on the extended Dynkin diagram.
These restrictions allow to exclude most of the ``wrong''
indices (Proposition~\ref{prop:listcomb}). Representation-theoretic arguments
allow to define Tits algebras of a semisimple group in the same fashion as this was done by Tits~\cite{Tits71}
over fields (Theorem~\ref{thm:algtits}). In Theorem~\ref{thm:levi} we give
a necessary and sufficient condition in terms of Tits algebras that a semisimple group scheme $H$
can be embedded into a larger semisimple group $G$ as the derived subgroup of a Levi subgroup of
a fixed parabolic subgroup of $G$.
Combining this result with the combinatorial restrictions, we obtain the list of all possible indices,
and show that the existence of a group with a given index is equivalent to the existence of an anisotropic group
(its anisotropic kernel) subject to certain explicitly stated conditions (\S~\ref{sec:ind}, Theorem~3).
In the field case, these conditions appeared earlier in~\cite{Selb,Tits90,Gar,TitsWeiss} for all Tits indices
except those where the $*$-action of the Galois group on the set of ``circled''
vertices is non-trivial.

Our proof of the existence of all indices of inner exceptional type (Theorem~\ref{th:ex})
is based on the knowledge of the list of maximal possible values of the $p$-relative canonical dimension of
the variety of Borel subgroups of simple algebraic groups corresponding to generic torsors. This list was
obtained in~\cite{PSZ} by means of the J-invariant of algebraic groups.

\section{Semisimple group schemes}

In this section we reproduce some definitions and results of~\cite{SGA}. Throughout the paper,
all references starting with Exp.~YZ are to this source.

Let $S$ be a scheme (not necessarily separated). A group scheme $G$ over $S$ is called \emph{reductive}
if it is affine and smooth over $S$, and its geometric fibers $G_{\overline{k(s)}}$ are connected
reductive groups
in the usual sense for all $s\in S$ (Exp.~XIX D\'ef.~4.7).
When $S$ is reduced, the smoothness can be replaced by
the condition that $G$ is finitely presented over $S$ and the dimension of a fiber is locally constant
(see Exp.~$\mathrm{VI_B}$, Cor.~4.4). The \emph{type} of $G$ at $s\in S$ is the root datum of $G_{\overline{k(s)}}$.
The type is locally constant (Exp.~XXII Prop.~2.8). To simplify the exposition,
in the sequel we consider
reductive group schemes of constant type only. Thus the type of a reductive group scheme $G$ is a root datum
$\RR=(\Phi,\Lambda,\Phi^*,\Lambda^*)$, where $\Phi$ is a root system, called
the {\it root system} of $G$, $\Lambda$ is a $\ZZ$-lattice
containing $\Phi$, called the {\it lattice of weights} of $G$, and $\Phi^*$ and $\Lambda^*$ are the dual
objects (Exp. XXI D\'ef. 1.1.1). A reductive group $G$ is {\it semisimple}, if the rank of $\Phi$ equals that of $\Lambda$.
We also usually include in the type a fixed subset of positive roots $\Phi^+$ in $\Phi$,
which determines a system of simple roots of $\Phi$ and, therefore, a Dynkin diagram $D$.

Over any scheme $S$ there exists a unique \emph{split} group scheme $G_0$ of a given type $\RR$, which actually comes from
a group scheme over $\Spec\ZZ$ known as the Chevalley -- Demazure group scheme
(Exp.~XXV Thm.~1.1). \emph{Quasi-split} group schemes over $S$ of the same type as $G_0$ are parametrized by
$\HH^1(S,\,\Aut(\RR,\,\Phi^+))$, where $\Aut(\RR,\,\Phi^+)$ is the group of automorphisms of $\RR$ preserving
$\Phi^+$ (cf. Exp.~XXIV Thm.~3.11).
All cohomology groups we consider are with respect to the fpqc
topology (but note that $\HH^1(S,\,H)=\HH^1_{\text{\'et}}(S,\,H)$ when $H$ is smooth).

Every semisimple group scheme $G$ is an inner twisted form of a uniquely determined quasi-split group $G_{qs}$,
given by a cocycle $\xi\in\HZ^1(S,\,G_{qs}^{ad})$, where $G_{qs}^{ad}$ is the adjoint group acting on $G_{qs}$ by
inner automorphisms. Cocycles in the same class in $\HH^1(S,\,G_{qs}^{ad})$ produce isomorphic group schemes
(Exp.~XXIV 3.12.1); however, distinct classes of cocycles may correspond to the same isomorphism class of groups
(see below).

A Dynkin diagram $D$ is nothing but a finite set of vertices together with a subset $E\subseteq D\times D$
of edges and a length function $D\to\{1,\,2,\,3\}$ (in other words, a colored graph). The scheme-theoretic
counterpart of this notion is called a \emph{Dynkin scheme} (Exp.~XXIV \S~3). So a Dynkin scheme
over $S$ is a twisted finite scheme $\DD$ over $S$ together with a subscheme $\mathcal{E}\subseteq\DD\times_S\DD$
and a map $\DD\to\{1,\,2,\,3\}_S$. Isomorphisms, base extensions and constant Dynkin schemes are
defined in an obvious way. We denote by $D_S$ the constant Dynkin scheme over $S$
corresponding to a Dynkin diagram $D$. By $\Aut(\DD)$ we always mean the scheme of automorphisms of $\DD$ over $S$
as a Dynkin scheme; it is a twisted constant group scheme over $S$.

To any semisimple group scheme $G$ one associates the Dynkin scheme $\Dyn(G)$. For a quasi-split group $\Dyn(G_{qs})$
is a twisted form of $D_S$ corresponding to the image in $\HZ^1(S,\,\Aut(D))$ of a cocycle
$\xi\in\HZ^1(S,\,\Aut(\RR,\,\Phi^+))$ defining $G_{qs}$ under the map induced by the canonical map
$\Aut(\RR,\,\Phi^+)\to\Aut(D)$ (Exp.~XXIV 3.7). When $G_{qs}$ is simply connected or adjoint, the latter map is an
isomorphism.

In general, $\Dyn(G)$ is isomorphic to $\Dyn(G_{qs})$, but the isomorphism is not canonical. By an \emph{orientation}
$u$ on $G$ we mean a choice of an element $u\in\Isomext(G_{qs},\,G)(S)$, that is of an isomorhism between
$\Dyn(G_{qs})$ and $\Dyn(G)$. A notion of an isomorphism of oriented group schemes is defined obviously. Exp.~XXIV Rem.~1.11
shows that $\HH^1(S,\,G_{qs}^{ad})$ is in bijective correspondence with the set of isomorphism classes of
\emph{oriented} inner twisted forms of $G_{qs}^{ad}$.

Let $T/S$ be a Galois covering that splits $\Dyn(G)$, i.e. $\Dyn(G)_T\simeq D_T$. For example,
one can take as $T$ the torsor corresponding to the cocycle in $\HZ^1(S,\,\Aut(D))$.
Every element $\sigma\in\Aut(T/S)$
acts on $\Dyn(G)_T$ and therefore defines some $\varphi_\sigma\in\Aut(D)(T)$ such that the diagram
$$
\xymatrix{
D_T\ar[d]\ar[r]^-{\varphi_\sigma}&D_T\ar[d]\\
T\ar[r]^-\sigma&T
}
$$
commutes. By Galois descent this action (which is called {\it the $*$-action}) completely determines
$\Dyn(G)$. If $S$ is connected, the $*$-action can be considered as an action of $\Aut(T/S)$ on the
Dynkin diagram $D$, and extends by $\QQ$-linearity to the $*$-action on $\Lambda$.

A subgroup scheme $P$ of $G$ is called \emph{parabolic} if it is smooth and $P_{\overline{k(s)}}$ is
a parabolic subgroup of $G_{\overline{k(s)}}$ in the usual sense for every $s\in S$  (Exp.~XXVI D\'ef.~1.1).
To a parabolic subgroup $P$ one can attach the \emph{type} $\tf(P)$ of $P$
which is a clopen subscheme of $\Dyn(G)$
(Exp.~XXVI 3.2). Note that the clopen subschemes of $\Dyn(G)$ are in one-to-one correspondence with the
$*$-invariant clopen subschemes of $D_T$, where $T/S$ is as above.

If $L$ is a Levi part of $P$, we have a canonical map $\Dyn(L)\to\Dyn(G)$ depending only on $L$ and $G$.
In particular, an orientation on $G$ \emph{induces} an orientation on $L$.

\section{Representation-theoretic lemmas}

By a \emph{representation} of a group scheme $G$ over $S$ we mean a homomorphism of algebraic groups
$\rho\colon G\to\GL_1(A)$, where $A$ is an Azumaya algebra (more formally, a sheaf of Azumaya algebras)
over $S$.

Let $G_0$ be a split semisimple group scheme over a scheme $S$, and let
$G_0\to\GL(V)$ be a representation of $G_0$ on a projective module
(more formally, a locally free sheaf of modules) $V$
of finite rank
over $S$.
Fix a maximal split torus $T_0$
of $G_0$ and let $\La$ and $\Lar$ be its lattices of weights and roots respectively.
Then $V$ decomposes into a direct sum $\bigoplus_{\lambda\in\La}V_\lambda$ so that for any scheme
$S'$ over $S$, any $t\in T_0(S')$, and any $v\in V_\lambda(S')$ one has $\rho(t)v=\lambda(t)v$
(Exp.~I Prop.~4.7.3). A character $\lambda$ with $V_\lambda\ne 0$ is called a \emph{weight} of $V$.

The \emph{cocenter} $\Cocent(G)$ of $G$ is the group scheme $\Hom(\Cent(G),\,\Gm)$. When $G$ is split it
can be identified with the constant group scheme $(\La/\Lar)_S$. Descent shows that
$\Cent(G)$ is isomorphic to $\Cent(G_{qs})$, and therefore $\Cocent(G)$ is isomorphic to $\Cocent(G_{qs})$.
The isomorphism depends only on the orientation of $G$.

A representation $\rho\colon G\to\GL_1(A)$ will be called \emph{center preserving} if
$\rho(\Cent(G))\subseteq\Cent(\GL_1(A))$.
In this case $\rho$ induces a homomorphism $\rho^{ad}\colon G^{ad}\to\PGL_1(A)$
and determines an element $\lambda_\rho\in\Cocent(G)(S)$, which is the restriction of
$\rho$ to $\Cent(G)$ composed with the natural isomorphism $\Cent(\GL_1(A))\simeq\Gm$.

\begin{lem}\label{lem:cent}
\begin{enumerate}
\item\label{item:cent} $G\to\GL(V)$ is center preserving if and only if over a splitting covering $\coprod S_\tau\to S$ every two weights of $V$ differ by an element of $\Lar$.
\item\label{item:dual} The dual $G\to\GL(V^*)$ of a center preserving representation $G\to\GL(V)$ is center preserving.
\item\label{item:tensor} The tensor product $G\to\GL(V_1\otimes V_2)$ of center preserving representations $G\to\GL(V_1)$ and $G\to\GL(V_2)$ is center preserving.
\item\label{item:summand} For any representation $\rho\colon G\to\GL(V)$ and an element $\lambda\in\Cocent(G)(S)$,
the submodule $W\subseteq V$ defined by
$$
W(S')=\{v\in V\times_SS'\mid c\cdot v=\lambda(c)v\text{ for all fpqc }S''/S'\text{ and }c\in\Cent(G)(S'')\}
$$
is a $G$-invariant direct summand of $V$. Moreover, the representation $\rho'\colon G\to\GL(W)$
is center preserving and $\lambda_{\rho'}=\lambda$ if $W\ne 0$.
\end{enumerate}
\end{lem}
\begin{proof}
For \eqref{item:cent} observe that since the condition $\rho(\Cent(G))\subseteq\Cent(\GL(V))$ is
local with respect to fpqc topology, we can assume that $G$ is split.
Then $V$ is center preserving if and only if restrictions of every two weights $\lambda$ and $\mu$
of $V$ to $\Cent(G)$ coincide. This means exactly that $\lambda-\mu$ belongs to $\Lar$ (Exp.~XXII Rem.~4.1.8).
Parts \eqref{item:dual} and \eqref{item:tensor} follow from \eqref{item:cent}.

To prove \eqref{item:summand}, define $W'(S')$ as the set of all $v\in V(S')$ such that there exist an fpqc covering $\coprod S'_\tau\to S'$
and, for each $\tau$, a finite number of elements $\lambda_1,\,\ldots,\,\lambda_k\in\Cocent(G)(S'_\tau)$ distinct from $\lambda$
and elements $v_1,\,\ldots,\,v_k\in V\times_SS'_\tau$ such that $v=v_1+\ldots+v_k$ and $cv_i=\lambda_i(c)v_i$ for all fpqc $S''_\tau/S'_\tau$
and $c\in\Cent(G)(S''_\tau)$. Obviously $W$ and $W'$ are $G$-invariant (sheaves of) submodules of $V$. Over a splitting covering of $G$ it is easily
seen that $V=W\oplus W'$; therefore it is also true over the base $S$. By construction the representation $\rho'\colon G_{qs}\to W$ is center
preserving and $\lambda_{\rho'}=\lambda$.
\end{proof}\medskip

\begin{lem}\label{lem:weightrepr} Let $G_{qs}$ be a quasi-split group over $S$. Then any element of $\Cocent(G_{qs})(S)$
appears as $\lambda_\rho$ for some center preserving representation $\rho\colon G_{qs}\to\GL(V)$.
\end{lem}
\begin{proof}
Over a splitting covering of $G_{qs}$ choose a weight $\lambda\in\Lambda$ that represents a given element of $\Cocent(G_{qs})(S)$.
Obviously $\lambda+\Lar$ is $*$-invariant.
It is known (see \cite[Ch. VI, Exerc. 5 du \S 2]{Bu}) that any weight is equivalent modulo $\Lar$ to
a minuscule weight. On the other hand, by \cite[3.1]{Tits71} we have $(\La/\Lar)^*=\La^*/\Lar^*$.
So we may assume that $\lambda$ is a $*$-invariant minuscule weight.

Consider first the split group $G_0$ over $\ZZ$. Recall briefly the construction of a Weyl module
$V(\lambda)$ for $G_0$ (see \cite{Jan} for details). We start from a finite dimensional
irreducible $(G_0)_{\CC}$-module
with the highest weight $\lambda$; we fix a vector $v_+$ of the weight $\lambda$ (which is unique up to
a scalar). Denote by $\UF$ the universal enveloping algebra of the Lie algebra of $(G_0)_{\CC}$, by $\UF^+$ and
$\UF^-$ its subalgebras generated by the positive (respectively, negative) root subspaces, and
by $\UF_{\ZZ}$, $\UF^+_{\ZZ}$, $\UF^-_{\ZZ}$ their $\ZZ$-forms used in the Chevalley's construction of split reductive groups. Then
$V(\lambda)$ is defined as $\UF^-_{\ZZ} v_+$. Note that $V(\lambda)$ is center preserving by Lemma~\ref{lem:cent}, \eqref{item:cent}.

Let $\Gamma$ be a subgroup of $\Aut(\RR,\,\Phi^+)$ preserving
$\lambda$. Then any element $\gamma\in\Gamma$ induces an automorphism of
$\UF_{\ZZ}$ which preserves $\UF^+_{\ZZ}$ and $\UF^-_{\ZZ}$. Since $\gamma$ preserves $\lambda$, the representations
$\rho\colon(G_0)_{\CC}\to\GL(V(\lambda)_{\CC})$ and $\rho\circ\gamma\colon(G_0)_{\CC}\to\GL(V(\lambda)_{\CC})$ are equivalent, and
their differentials are equivalent as well. Therefore, there exists $\varphi\in\GL(V(\lambda)_{\CC})$ such that
$\gamma(g)\varphi(v)=\varphi(gv)$ for every $v\in V(\lambda)_{\CC}$ and $g\in\UF$; moreover, $\varphi$ is unique up to a scalar.
It is easy to see that $\varphi$ preserves the
line spanned by $v_+$, and we can normalize
$\varphi$ so that $\varphi(v_+)=v_+$.
Now,
$$
\varphi(\UF^-_{\ZZ}v_+)\le\gamma(\UF^-_{\ZZ})\varphi(v_+)=\UF^-_{\ZZ}v_+,
$$
so $\varphi$ induces an automorphism $\varphi_{\ZZ}$ of $V(\lambda)$ compatible with $\gamma$ and preserving $v_+$. Since $\ZZ[G_0]$
is a Hopf subalgebra of ${\mathbb Q}[G_0]$ and $V(\lambda)$ is a subcomodule of $V(\lambda)_{\mathbb Q}$, and $\CC/{\mathbb Q}$ is faithfully flat,
$\varphi_{\ZZ}$ is an equivalence of the representations
$\rho\colon G_0\to\GL(V(\lambda))$ and $\rho\circ\gamma\colon G_0\to\GL(V(\lambda))$. Moreover, since $\varphi_{\ZZ}$ is uniquely
determined by $\gamma$, we obtain a homomorphism $\psi\colon\Gamma\to\GL(V(\lambda))$.

Now let $\xi$ be a cocycle in $\HZ^1(S,\,\Gamma)$ producing $G_{qs}$.
The cocycle $\psi_*(\xi)$ then defines a projective module $V$ together with a representation $G_{qs}\to\GL(V)$ we need.
\end{proof}\medskip

\section{Tits algebras}

\begin{thm}\label{thm:algtits}
Let $(G,\,u)$ be an oriented semisimple group scheme of constant type over $S$ corresponding
to the class $[\xi]\in\HH^1(S,\,G^{ad}_{qs})$.
\begin{enumerate}
\item There exist two natural mutually quasi-inverse equivalences $F_u$, $F'_u$ between
the categories of group schemes over $S$ with $G_{qs}^{ad}$-action (by
group automorphisms) and group schemes over $S$ with $G^{ad}$-action.
In particular, each center preserving representation $\rho\colon G_{qs}\to\GL(V)$ gives rise to
a center preserving representation $F_u(\rho)\colon G\to\GL_1(A_{u,\,\rho})$ for some Azumaya algebra
$A_{u,\,\rho}$.

\item\label{item:brauer} The class $[A_{u,\,\rho}]$ in the Brauer group $\Br(S)$ depends only on $\lambda_{F_u(\rho)}$
and not on the particular choice of $u$ and $\rho$. Its image in $\HH^2(S,\,\Gm)$
coincides with $(\lambda_{\rho})_*\delta([\xi])$, where
$$
(\lambda_\rho)_*\colon\HH^2(S,\,\Cent(G_{qs}))\to\HH^2(S,\,\Gm),
$$
and $\delta$ is the connecting homomorphism in the long exact sequence arising from the sequence
$$
\xymatrix{
1\ar[r]&\Cent(G_{qs})\ar[r]&G_{qs}\ar[r]&G^{ad}_{qs}\ar[r]&1.
}
$$
\end{enumerate}
\end{thm}
\begin{proof}
{\bf 1.} Consider the left $G^{ad}$- and right $G^{ad}_{qs}$-torsor $I=\Isomint_u(G_{qs},\,G)$
(see Exp.~XXIV Rem.~1.11). Let $H$ be a group scheme with a $G^{ad}_{qs}$-action.
Then $F_u(H)=I\times^{G^{ad}_{qs}}H$ is a group scheme over $I/G^{ad}_{qs}\simeq S$
with a left $G^{ad}$-action. Similarly, $F'_u$ is defined by
$F'_u(H')=I'\times^{G^{ad}}H'$, where $I'=\Isomint_{u^{-1}}(G,\,G_{qs})$. Further,
we have isomorphisms
$I'\times^{G^{ad}}I\simeq G^{ad}_{qs}$ and $I\times^{G^{ad}_{qs}}I'\simeq G^{ad}$, hence $F_u$ and $F'_u$ are mutually quasi-inverse.

{\bf 2.} The cohomological class in $\HH^1(S,\,\PGL(V))$ corresponding to $A_{u,\,\rho}$
is nothing but $\rho^{ad}_*([\xi])$, where $\rho^{ad}\colon G^{ad}_{qs}\to \PGL(V)$ is the representation
induced by $\rho$.
Now the last assertion of the Theorem follows from the commutativity of the diagram
$$
\xymatrix{
\HH^1(S,\,G^{ad}_{qs})\ar[d]_{\rho^{ad}_*}\ar[r]^-{\delta}&\HH^2(S,\,\Cent(G_{qs}))\ar[d]^{(\lambda_\rho)_*}\\
\HH^1(S,\,\PGL(V))\ar[r]&\HH^2(S,\,\Gm),
}
$$
which comes from the diagram
$$
\xymatrix{
1\ar[r]&\Cent(G_{qs})\ar[d]_{\lambda_\rho}\ar[r]&G_{qs}\ar[d]_\rho\ar[r]&G^{ad}_{qs}\ar[d]_{\rho^{ad}}\ar[r]&1\\
1\ar[r]&\Gm\ar[r]&\GL(V)\ar[r]&\PGL(V)\ar[r]&1.
}
$$
Thus, once $u$ is fixed, the class of $A_{u,\,\rho}$ depends only on $\lambda_\rho$. Let $v$ be another
orientation on $G$. Then $\rho'=F'_v(F_u(\rho))$ is $\rho$ composed with the corresponding outer automorphism
of $G_{qs}$; in particular, its target is still $\GL(V)$. Obviosly $F_{v}(\rho')\simeq F_u(\rho)$. Now, if
$\sigma$ is another representation of $G_{qs}$ with $\lambda_{F_v(\sigma)}=\lambda_{F_u(\rho)}$, then
$$
\lambda_\sigma=\lambda_{F_v(\sigma)}\circ v=\lambda_{F_u(\rho)}\circ v=\lambda_{F_v(\rho')}\circ v=\lambda_{\rho'},
$$
hence
$$
[A_{v,\,\sigma}]=[A_{v,\,\rho'}]=[A_{u,\,\rho}].
$$
\end{proof}\medskip

The Azumaya algebra $A_{u,\,\rho}$ will be called the \emph{Tits algebra} of $G$ corresponding to
a center preserving representation $\rho\colon G_{qs}\to\GL(V)$. We denote by $\beta_G$ the homomorphism
\begin{align*}
\beta_G\colon\Cocent(G)(S)&\to\Br(S)\\
\lambda&\mapsto[A_{u,\,\rho}]\text{ with }\lambda_{F_u(\rho)}=\lambda.
\end{align*}
It is well-defined in view of Lemma~\ref{lem:weightrepr} and Theorem~\ref{thm:algtits}. To see that $\beta_G$ is indeed a homomorphism one can use either the
tensor product of representations or the fact that $\Br(S)$ is a subgroup in $\HH^2(S,\,\Gm)$.

If the orientation $u$ is fixed, we will consider $\beta_G$ as a homomorphism from $\Cocent(G_{qs})$ to $\Br(S)$.
Further, for an element $\lambda$ of $\La^*$ we will write $\beta_G(\lambda)$ instead of $\beta_G(\lambda|_{\Cent(G_{qs})})$.

\medskip

The Dynkin scheme $\Dyn(G)$ is the disjoint union of its minimal clopen subschemes which will be called \emph{orbits} for brevity;
they indeed correspond to orbits of the $*$-action on a set of simple roots.

Assume that $G$ is simply connected. Let $T_{qs}$ be a fixed maximal torus of $G_{qs}$, $T_{qs}^{ad}$ be the respective torus in $G_{qs}^{ad}$.
Over a splitting covering we have two canonical homomorphisms
\begin{align*}
&\omega\colon\Dyn(G)\to\Hom(T_{qs},\Gm),\\
&\alpha\colon\Dyn(G)\to\Hom(T_{qs}^{ad},\Gm),
\end{align*}
that associate to each vertex $i$ of the Dynkin diagram the fundamental weight $\omega_i$ or, respectively,
the simple root $\alpha_i$. By faithfully flat descent these homomorphisms are defined over the base scheme $S$.

Let $O$ be an orbit in $\Dyn(G)$. Composing $\omega$ (resp., $\alpha$) with the inclusion $O\to\Dyn(G)$, we obtain
a weight $\omega_O\colon (T_{qs})_O\to\Gm$ (resp., a root $\alpha_O\colon (T_{qs}^{ad})_O\to\Gm$),
which will be called the \emph{canonical weight}
(resp., the {\it canonical root}) corresponding to $O$ (cf. Exp. XXIV 3.8). It is easy to see that $\alpha_O$ and $\omega_O$
are $*$-invariant weights of $G_O$.

Recall that the Weil restriction $R_{S'/S}$ ($\prod_{S'/S}$ in the notation of~\cite{SGA}) is the right adjoint to the base change
functor. So we have homomorphisms
\begin{align*}
&\bar\omega_O\colon T_{qs}\to R_{O/S}(\Gm),\\
&\bar\alpha_O\colon T_{qs}^{ad}\to R_{O/S}(\Gm).
\end{align*}

If $O$ splits over an extension $S'/S$ into a disjoint union $\coprod_i O_i$, then $(\bar\omega_O)_{S'}$ (resp. $(\bar\alpha_O)_{S'}$)
is equal to $\prod_i\omega_{O_i}$ (resp., $\prod_i\alpha_{O_i}$) composed with the natural isomorphism
$\prod_i R_{O_i/S'}(\Gm)\simeq R_{\coprod_i O_i/S'}(\Gm)$. In particular, over a splitting covering $\bar\omega_O$ (resp. $\bar\alpha_O$)
can be identified with an appropriate product of $\omega_i$ (resp., $\alpha_i$).

\begin{prop}\label{prop:canweight}
\begin{enumerate}
\item\label{item:qstor} In the above setting we have the isomorphism
$$
\prod_O\bar\omega_O\colon T_{qs}\simeq \prod_O R_{O/S}(\Gm)
$$
(cf. Exp.~XXIV Prop.~3.13).

\item\label{item:levicent} If $L'_{qs}$ is the standard Levi subgroup of a standard parabolic subgroup $P$ in $G_{qs}^{ad}$, then we have
the isomorphism
$$
\prod_{O\colon O\not\subset\tf(P)}\bar\alpha_O\colon\Cent(L'_{qs})\simeq\prod_{O\colon O\not\subset\tf(P)}R_{O/S}(\Gm).
$$

\item\label{item:qslevi} We have
$$
L'_{qs}=\Cent_{G_{qs}}(Q)=\Cent_{G_{qs}}(Q_{diag}),
$$
where $Q$ is the natural split subtorus
$\prod_{O\colon O\not\subset\tf(P)}\Gm$ of\, $\prod_{O\colon O\not\subset\tf(P)}R_{O/S}(\Gm)$, and $Q_{diag}$
is the split torus of rank $1$ embedded diagonally into $Q$.
\end{enumerate}
\end{prop}
\begin{proof}
Let's prove \eqref{item:levicent}. Note that $\Cent(L'_{qs})$ is contained in $T_{qs}^{ad}$, so the map is well-defined.
Over each element $S_\tau$ of a splitting covering of $S$ the Dynkin scheme can be identified with
a set $D$
and $\tf(P)$ with a subset $D\setminus J$. The map $\prod_{O\colon O\not\subset\tf(P)}\bar\alpha_O$ becomes $\prod_{i\in J}\alpha_i$,
and $\Cent(L'_{qs})_{S_\tau}$ equals $\bigcap_{i\in D\setminus J}\Ker\alpha_i$. But
$$
\prod_{i\in D}\alpha_i\colon (T_{qs}^{ad})_{S_\tau}\to\prod_{i\in D}\Gm
$$
is an isomorphism, and \eqref{item:levicent} follows. Part \eqref{item:qstor} can be proved similarly and even easier.

We have obvious inclusions
$$
L'_{qs}\le\Cent_{G_{qs}}(Q)\le\Cent_{G_{qs}}(Q_{diag}),
$$
so to prove \eqref{item:qslevi} it suffices to show that $H=\Cent_{G_{qs}}(Q_{diag})$ is contained in $L'_{qs}$. We can pass to a splitting covering.
By Exp.~XXVI Prop.~6.1 $H_{S_\tau}$ is smooth with connected fibers; clearly it contains $(T_{qs}^{ad})_{S_\tau}$.
By Exp.~XXII 5.4.1 such subgroup is uniquely determined by the set of roots $\alpha$ such that the generator $e_\alpha$ of $\Lie((G_{qs})_{S_\tau})$
is contained in its Lie algebra. Note that the restriction of a simple root $\alpha_i$ to $Q_{diag}$ is identity when $i\in J$ and is trivial otherwise.
So $e_\alpha$ belongs to $\Lie(H_{S_\tau})$ if and only if the sum of its coefficients at $\alpha_i$ with $i\in J$ is zero. But $(L'_{qs})_{S_\tau}$
is also smooth with connected fibers and corresponds to the same set of roots, hence $L'_{qs}=H$.
\end{proof}\medskip

\begin{prop}\label{prop:strin}
In the setting of Theorem~{\rm\ref{thm:algtits}}, assume moreover that $G$ is simply connected
and $\Pic(\Dyn(G))=0$. Then $[\xi]$ comes from an element in $\HH^1(S,\,G_{qs})$ if and only if $\beta_{G_O}(\omega_O)=0$
for each orbit $O$.
\end{prop}
\begin{proof}
If $[\xi]$ belongs to the image of $\HH^1(S,\,G_{qs})\to\HH^1(S,\,G^{ad}_{qs})$ then $\delta([\xi]_O)=0$ and
therefore $\beta_{G_O}=0$ for each $O$. Conversely, assume that $\beta_{G_O}(\omega_O)=0$ for each $O$.
Proposition~\ref{prop:canweight} applied to the Borel subgroup implies that $T_{qs}\simeq\prod_O R_{O/S}(\Gm)$
and $T_{qs}^{ad}\simeq\prod_O R_{O/S}(\Gm)$. Now the Shapiro lemma (cf. Exp. XXIV Prop. 8.2) implies that the image of $\delta([\xi])$ in
$\HH^2(S,\,T_{qs})$ is trivial, while $\HH^1(S,\,T^{ad}_{qs})=\Pic(\Dyn(G))=0$. Now
the claim follows from the exact sequence
$$
\xymatrix{
\HH^1(S,\,T^{ad}_{qs})\ar[r]&\HH^2(S,\,\Cent(G_{qs}))\ar[r]&\HH^2(S,\,T_{qs}),
}
$$
which comes from the sequence
$$
\xymatrix{
1\ar[r]&\Cent(G_{qs})\ar[r]&T_{qs}\ar[r]&T^{ad}_{qs}\ar[r]&1.
}
$$
\end{proof}\medskip

\begin{thm}\label{thm:levi}
\begin{enumerate}
\item\label{item:titsrestr} Let $(G,\,u)$ be an oriented semisimple group scheme of constant type over $S$,
$P$ be its parabolic subgroup admitting a Levi subgroup $L$, $H$ be the derived subgroup of $L$ with the induced orientation.
Denote by $\La$ the lattice of weights of $G_{qs}$. For every $\lambda\in\La^*$ denote by $\lambda'$ the restriction of $\lambda$ to
$\Cent(H_{qs})$. Then $\beta_G(\lambda)=\beta_H(\lambda')$. In particular, for any $\alpha\in\Lar^*$ one has
$\beta_H(\alpha')=0$.

\item\label{item:leviembed} Let $G_{qs}$ be a quasi-split simply connected group, $P_{qs}$ be a standard parabolic subgroup
of $G_{qs}$, $L_{qs}$ be its standard Levi subgroup, $H_{qs}$ be the derived subgroup of $L_{qs}$. Assume that $(H,\,v)$ is an
oriented inner form of $H_{qs}$, satisfying $\beta_{H_O}(\alpha'_O)=0$ for all $O\not\subset\tf(P_{qs})$. Then there exist an
oriented inner form $(G,\,v)$ of $G_{qs}$ and its parabolic subgroup $P$ admitting a Levi subgroup $L$ such that
the derived subgroup of $L$ with the induced orientation is isomorphic to $H$.

\item\label{item:unique} In the setting of {\rm\eqref{item:leviembed}}, assume that $\Pic(\Dyn(S))=0$. Then $(G,\,u)$ is unique up to an isomorphism.

\item\label{item:bij} In the setting of {\rm\eqref{item:leviembed}}, assume that $S$ is semilocal. Then $(G,\,u)$ determines $(H,\,v)$ up to an isomorphism.
\end{enumerate}
\end{thm}
\begin{proof}

{\bf 1.} Let $\xi$ be a cocycle in $\HZ^1(S,\,G^{ad}_{qs})$ corresponding to $G$, given by elements
$g_{\sigma\tau}\in G^{ad}_{qs}(S_\sigma\times_SS_\tau)$ for some covering $\coprod S_\tau\to S$ that quasi-splits $G$.
Over each $S_\tau$ one can (possibly, passing to a finer covering) conjugate $P_{S_\tau}$ and $L_{S_\tau}$ by some
element of $G^{ad}_{qs}$ to $P_{qs}$ and $L_{qs}$, where $P_{qs}$ is a standard parabolic subgroup of $G_{qs}$ and
$L_{qs}$ is its standard Levi subgroup. Adjusting $\xi$ by the coboundary given by these elements,
we can assume that all $g_{\sigma\tau}$'s belong to $L'_{qs}$, where $L'_{qs}$ is the image of $L_{qs}$ in $G_{qs}^{ad}$,
by Exp.~XXVI Prop.~1.15 and Cor.~1.8 (cf. Exp.~XXVI 3.21)

Let $\rho\colon G_{qs}\to\GL(V)$ be a center preserving representation with a weight $\lambda$. Consider its
restriction to $H_{qs}$ and denote by $U$ the center preserving direct summand corresponding to $\lambda'$
and by $U'$ its complement invariant under $H_{qs}$ (see Lemma~\ref{lem:cent}, \eqref{item:summand}). Denote by $T_{qs}$ the standard
maximal torus of $L_{qs}$ and by $T'_{qs}$ its intersection with $H_{qs}$. Note that $U$ and $U'$, being sums of weight
subspaces of $T'_{qs}$, are stable under $T_{qs}$ and, therefore, are invariant under the action of $L_{qs}$.
Hence the map $\HH^1(S,\,L'_{qs})\to\HH^1(S,\,\PGL(V))$ factors through $\HH^1(S,\,(\GL(U)\times\GL(U'))/\Gm)$,
where $\Gm$ is embedded into $\GL(U)\times\GL(U')$ diagonally.

Now the claim is obtained by comparing the diagrams
$$
\xymatrix@C=30pt{
\HH^1(S,\,(\GL(U)\times\GL(U'))/\Gm)\ar[r]\ar[d]&\HH^2(S,\,\Gm)\ar@{=}[d]\\
\HH^1(S,\,\PGL(V))\ar[r]&\HH^2(S,\,\Gm)
}
$$
and
$$
\xymatrix@C=30pt{
\HH^1(S,\,(\GL(U)\times\GL(U'))/\Gm)\ar[r]\ar[d]&\HH^2(S,\,\Gm)\ar@{=}[d]\\
\HH^1(S,\,\PGL(U))\ar[r]&\HH^2(S,\,\Gm),
}
$$
which come from the sequences
$$
\xymatrix@C=30pt{
1\ar[r]&\Gm\ar[r]\ar@{=}[d]&\GL(U)\times\GL(U')\ar[r]\ar[d]&(\GL(U)\times\GL(U'))/\Gm\ar[d]\ar[r]&1\\
1\ar[r]&\Gm\ar[r]&\GL(V)\ar[r]&\PGL(V)\ar[r]&1.
}
$$
and
$$
\xymatrix@C=30pt{
1\ar[r]&\Gm\ar[r]\ar@{=}[d]&\GL(U)\times\GL(U')\ar[r]\ar[d]&(\GL(U)\times\GL(U'))/\Gm\ar[d]\ar[r]&1\\
1\ar[r]&\Gm\ar[r]&\GL(U)\ar[r]&\PGL(U)\ar[r]&1.
}
$$

{\bf 2.} Let $[\zeta]$ be the class in $\HH^1(S,\,H^{ad}_{qs})=\HH^1(S,\,L^{ad}_{qs})$ corresponding to $H$. Denote by $L'_{qs}$
and $H'_{qs}$ the images of $L_{qs}$ and $H_{qs}$ in $G^{ad}_{qs}$. Let us compute the image $\delta([\zeta])\in\HH^2(S,\,\Cent(L'_{qs}))$.
Using the assumption, Theorem~\ref{thm:algtits} \eqref{item:brauer}, and the commutative diagram
$$
\xymatrix{
\HH^1(S,\,H^{ad}_{qs})\ar@{=}[d]\ar[r]^-\delta&\HH^2(S,\,\Cent(H'_{qs}))\ar[d]\\
\HH^1(S,\,L^{ad}_{qs})\ar[r]^-\delta&\HH^2(S,\,\Cent(L'_{qs})),
}
$$
we see that $(\alpha_O)_*\delta([\zeta_O]))=0$ for any $O\not\subset\tf(P_{qs})$. Now Proposition~\ref{prop:canweight} \eqref{item:levicent}
and the Shapiro lemma show that $\delta([\zeta])=0$. It means that $[\zeta]$ comes from some $[\xi]\in\HH^1(S,\,L'_{qs})$,
and the image of $[\xi]$ in $\HH^1(S,\,G^{ad}_{qs})$ defines the desired $G$.

{\bf 3.} Let $(G,\,u)$ be such a group; denote by $[\xi]$ the corresponding class in $\HH^1(S,\,G^{ad}_{qs})$.
As we have seen in the proof of \eqref{item:titsrestr}, $[\xi]$ comes from an element of $\HH^1(S,\,L'_{qs})$, say $[\zeta]$. We have to show
that $[\zeta]$ (and a fortioti $[\xi]$) is completely determined by its image in $\HH^1(S,\,L^{ad}_{qs})$, or,
in other words, that the canonical map $\pi_*\colon\HH^1(S,\,L'_{qs})\to\HH^1(S,\,L^{ad}_{qs})$ is injective. Since $\Cent(L'_{qs})$
is central in $L'_{qs}$, $\HH^1(S,\,\Cent(L'_{qs}))$ acts on $\HH^1(S,\,L'_{qs})$, and the orbits of the action coincide with
the fibers of $\pi_*$. But $\HH^1(S,\,\Cent(L'_{qs}))$ by Proposition~\ref{prop:canweight} \eqref{item:levicent} and the Shapiro lemma
injects into $\Pic(\Dyn(G))$, which is trivial by the assumption.

{\bf 4.} Follows from the proof of \eqref{item:unique} and the fact that the map $\HH^1(S,\,L'_{qs})\to\HH^1(S,\,G^{ad}_{qs})$ is
injective (Exp.~XXVI Cor.~5.10).
\end{proof}\medskip

\section{Combinatorial restrictions}\label{sec:comb}

From now on we assume that $S=\Spec R$, where $R$ is a connected semilocal ring.
Recall that in this case
all minimal parabolic subgroups $P_{\min}$ of $G$ are conjugate under $G(S)$ and hence have the same type
$\tf_{\min}=\tf(P_{\min})$, which is a clopen subscheme of $\Dyn(G)$ (Exp.~XXVI Cor.~5.7).
 By Exp. XXVI Lemme 3.8 $P\mapsto \tf(P)$ is a bijection between
parabolic subgroups $P$ of $G$ containing $P_{\min}$ and clopen subschemes
$\tf$ of $\Dyn(G)$ containing $\tf_{\min}$.

Since $S$ is affine, for any parabolic subgroup $P$ of $G$ there exists
a Levi subgroup $L$ (Exp.~XXVI Cor.~2.3) of $P$, and a
unique parabolic subgroup $P^-$ which is opposite to $P$ with respect to $L$, i.e.
satisfies $P^-\cap P=L$ (Exp.~XXVI Th.~4.3.2). The type $\tf(P^-)$ is the image $s_G(\tf(P))$
of $\tf(P)$ under
an automorphism $s_G$ of $\Dyn(G)$ called the {\it opposition involution} (Exp.~XXIV
Prop.~3.16.6 and Exp.~XXVI 4.3.1; cf.~\cite{Tits66} 1.5.1). The corresponding automorphism
$s_G\in\Aut(D)$ is induced by the automorphism $\alpha\mapsto -w_0(\alpha)$ of the root system $\Phi$
of $G_0$, where $w_0$ is the unique element of maximal length in the Weyl group of $\Phi$.
In fact $s_G$ acts nontrivially only on irreducible components of $\Phi$ of type $A_n$, $n\ge 2$, $D_{2n+1}$,
$n\ge 1$, or $E_6$, where it coincides with the unique nontrivial automorphism of the component.

By the \emph{Tits index} of $G$ we mean the pair $(\Dyn(G),\,\tf_{\min})$.
Clearly, we have $\tf_{\min}=s_G(\tf_{\min})$, since if $P=P_{\min}$ is a minimal parabolic subgroup,
then $P^-$ is also minimal.

The group $G$ is quasi-split if $\tf_{\min}$ is empty. In the opposite case
when $\tf_{\min}=\Dyn(G)$ we say that $G$ is \emph{anisotropic}. The
\emph{anisotropic kernel} $G_{an}$ of $G$ is defined as the derived subgroup
of a Levi part of $P_{\min}$, which is indeed anisotropic by Exp.~XXVI Prop.~1.20.

Tits indices can be described in the set-theoretic style as follows. The assumption that $S=\Spec R$
is connected allows us to identify $D_T$ with $D$, and a clopen $*$-invariant subscheme of $D_T$ with a $*$-invariant
subset of $D$. Let $J\subseteq D$ be the complement of the subset corresponding to $(\tf_{\min})_T$.
Then the Tits index of $G$ is determined by the pair $(D,\,J)$ together with a $*$-action on $D$, represented
by a subgroup $\Gamma$ of $\Aut(D)$. Usually we indicate $\Gamma$ by writing its order as the upper left index attached to $D$
(for example, ${}^2E_6$, ${}^6D_4$ and so on). The group $G$ is \emph{of inner type} if $\Dyn(G)\simeq\Dyn(G_0)$,
or, in other words, $\Gamma=\{1\}$.

From now on, we fix a minimal parabolic subgroup $P=P_{\min}$ of $G$, a Levi subgroup
$L$ of $P$, and a maximal split subtorus $Q$ of $G$ such that $L=\Cent_G(Q)$, which exists by
Exp. XXVI Cor.~6.11 (or by Proposition~\ref{prop:canweight}~\eqref{item:qslevi} and descent).
Let $M$ be the lattice of characters of $Q$.
The Lie algebra $\Lie(G)$ of $G$ decomposes under the action of $Q$
into a direct sum of weight subspaces:
$$
\Lie(G)=\Lie(L)\oplus\bigoplus_{\alpha\in M\setminus\{0\}}\Lie(G)^{\alpha}.
$$
We denote by $\Psi$ the set of elements $\alpha\in M\setminus\{0\}$ such that $\Lie(G)^{\alpha}\neq 0$.
By Exp. XXVI Th. 7.4 $\Psi$ is a root system, which is called the \emph{relative root system} of $G$ with respect to $Q$.
One readily sees that the simple roots of $\Psi$ correspond bijectively to the $*$-orbits contained in $J$.

Denote by $\hat D$ the extended Dynkin diagram (one adds a vertex corresponding to minus
the maximal root to each irreducible component of $D$), and by $\hat J$ the union $J\cup(\hat D\setminus D)$.

\begin{lem}\label{lem:opinv}
Let $G$ be a semisimple algebraic group over $S$, and let $(D,\,J)$ be the Tits index of $G$.
Then any $*$-orbit $O\subseteq\hat J$ is invariant
under the opposition involution of the Dynkin diagram $(\hat D\setminus\hat J)\cup O$.
\end{lem}
\begin{proof}
Let $A\in\Psi$ be the relative root corresponding to $O$ (it is simple if $O\subseteq J$ and the opposite to
the maximal otherwise). By Exp.~XXVI Prop.~6.1 the subsets $\ZZ A\cap\Psi$ and $\ZZ A\cap\Psi^+$ correspond
to certain subgroups $G'$ and $P'$ of $G$; moreover, $G'$ is reductive and $P'$ is a parabolic subgroup of $G'$
having $L$ as a Levi subgroup. Since $L$ is anisotropic, $P'$ is a minimal parabolic subgroup of $G'$.
Passing to a splitting covering one sees that the Dynkin diagram of $G'$ is $(\hat D\setminus\hat J)\cup O$,
and the type of $P'$ is given by $O$. The Lemma follows.
\end{proof}\medskip

In the next Proposition we list all possible cases when the conclusion of Lemma~\ref{lem:opinv} holds
for an irreducible root system $\Phi$. Our numbering of the vertices of Dynkin diagrams follows~\cite{Bu}.

\begin{prop}\label{prop:listcomb}
Let $\Phi$ be a reduced irreducible root system, $D$
the corresponding Dynkin diagram, $J\neq\emptyset$ a subset of $D$ and $\Gamma$ a group of automorphisms of $D$.
A triple $(\Phi,J,\Gamma)$ satisfies that
any $\Gamma$-orbit $O\subseteq \hat J$ is invariant
under the opposition involution of the Dynkin diagram $(\hat D\setminus \hat J)\cup O$,
if and only if it is, up to an automorphism of $D$, one in the following list:
\begin{enumerate}
\item $\Phi=A_n$, $n\ge 1$; $|\Gamma|=1$; $J=\{d,\,2d,\,\ldots,\,rd\}$ for some $d,r\ge 1$ such that
$d\cdot(r+1)=n+1$.

\item $\Phi=A_n$, $n\ge 2$; $|\Gamma|=2$; $J=\{d,\,2d,\,\ldots,\,rd,\,n+1-d,n+1-2d,\ldots,n+1-rd\}$
for some $d,r\ge 1$ such that $d\mid n+1$, $2rd\le n+1$.

\item $\Phi=B_n$, $n\ge 2$; $|\Gamma|=1$; $J=\{d,\,2d,\,\ldots,\,rd\}$ for some $d,r\ge 1$
such that $d$ is even or $d=1$, $rd\le n$.

\item $\Phi=C_n$, $n\ge 2$; $|\Gamma|=2$; $J=\{d,\,2d,\,\ldots,\,rd\}$ for some $d,r\ge 1$
such that $rd\le n$.

\item $\Phi=D_n$, $n\ge 4$; $|\Gamma|=1$; $J=\{d,\,2d,\,\ldots,\,rd\}$ for some $d,r\ge 1$
such that $d$ is even or $d=1$, $rd\le n$, $rd\neq n-1$.

\item $\Phi=D_n$, $n\ge 4$; $|\Gamma|=2$; $J=\{d,\,2d,\,\ldots,\,rd\}$ (or $J=\{d,2d,\ldots,(r-2)d,n-1,n\}$ in the case $rd=n-1$)
for some $d,r\ge 1$ such that $d$ is even or $d=1$, $rd\le n-1$.

\item $\Phi=D_4$; $|\Gamma|=3$ or $|\Gamma|=6$; $J=\{2\},\,D$.

\item $\Phi=E_6$; $|\Gamma|=1$; $J=\{2\},\,\{1,\,6\},\,\{2,\,4\},\,D$.

\item $\Phi=E_6$; $|\Gamma|=2$; $J=\{2\},\,\{1,\,6\},\,\{2,\,4\},\,\{1,\,6,\,2\},\,D$.

\item $\Phi=E_7$; $|\Gamma|=1$; $J=\{1\},\,\{6\},\,\{7\},\,\{1,\,3\},\,\{1,6\},\,\{1,6,7\},\,\{1,3,4,6\},\,D$.

\item $\Phi=E_8$; $|\Gamma|=1$; $J=\{1\},\,\{8\},\,\{1,\,8\},\,\{7,\,8\},\,\{1,\,6,\,7,\,8\},\,D$.

\item $\Phi=F_4$; $|\Gamma|=1$; $J=\{1\},\,\{4\},\,\{1,\,4\},\,D$.

\item $\Phi=G_2$; $|\Gamma|=1$; $J=\{2\},\,D$.
\end{enumerate}
\end{prop}
\begin{proof}
If $\Phi$ is an exceptional root system or $D_4$, the result is verified by an easy
try-out. Consider the case $\Phi=A_n$, $|\Gamma|=1$. The opposition
involution of $A_n$ is the non-trivial automorphism of $D$, hence if $|J|=1$ then $n=2k+1$ and $J=\{k+1\}$, the
middle vertex. Proceeding by induction on $|J|$, we see that $J=\{d,\,2d,\,\ldots,\,rd\}$ for some
$d\ge 1$ such that $d|n+1$, $d\cdot(r+1)=n+1$, and any such $J$ is valid. If $|\Gamma|=2$ then since $J$ is $\Gamma$-invariant, $J$ contains
a vertex $k$ if and only if it contains $n+1-k$; the opposition involution condition implies that
$J=\{d,\,2d,\,\ldots,\,rd\}\cup\{n+1-d,\,n+1-2d,\,\ldots,\,n+1-rd\}$, and any such $J$ is valid.

Now consider the case $\Phi=B_n,\,C_n,\,D_n$ and $|\Gamma|=1$. Let $J=\{i_1,\,i_2,\,\ldots,\,i_r\}$, $i_1<i_2<\ldots<i_r$.
If $\Phi=D_n$ and $i_r>n-2$, we may assume $i_r=n$ applying an automorphism of $D$.
Then $J\setminus\{i_r\}$ lies in the connected component of $D\setminus\{i_r\}$
of type $A_{i_r-1}$. Since $J\setminus\{i_r\}$ satisfies the opposition involution condition, by the $A_n$ case
$J\setminus\{i_r\}$ is of the form $\{d,\,2d,\,\ldots,\,(r-1)d\}$ for some $d\ge 1$ such that $i_r=rd$.
Therefore, $J=\{d,\,2d,\,\ldots,\,rd\}$, as required. If $\Phi=C_n$, this finishes the proof,
since any such $J$ satisfies the opposition involution condition.
If $\Phi=D_n$ or $B_n$, such $J$ does not satisfy the opposition involution condition for
$O=\hat J\setminus J$ if $d$ is odd $>1$, so this case is excluded. The case $\Phi=D_n$, $|\Gamma|=2$ is verified analogously.
\end{proof}\medskip

\section{Tits indices}\label{sec:ind}

We now start the classification of semisimple algebraic groups over $S=\Spec R$, where $R$ is
a connected semilocal ring. The problem allows two immediate reductions.
First, every semisimple group $G$ is completely determined by its root datum and the corresponding simply connected
group $G^{sc}$, so we can assume that $G$ is simply connected.

Second, if the Dynkin diagram $D$ of $G$ is not connected (that is, the root system is not irreducible), we can present $D$ as the disjoint union of its \emph{isotypic} components
$D_t$ (it means that we collect isomorphic components together), and then we have a canonical decomposition
$G\simeq\prod G_t$, where $G_t$ is a group over $S$ with the Dynkin diagram $D_t$ (Exp.~XXIV Prop.~5.5).
Further, if $D_t$ is the disjoint union of $n_t$ copies of a connected graph $D_{0,\,t}$, there exists
a canonical \'etale extension $S_t/S$ of degree $n_t$ and a group $G_{0,\,t}$ over $S_t$ such that
$G_t\simeq R_{S_t/S}(G_{0,\,t})$ (Exp.~XXIV Prop.~5.9). So we can assume that $D$ is connected,
that is, $G$ is a {\it simple} algebraic group.

Our reasoning will be based on Theorem~\ref{thm:levi}, which implies that an oriented semisimple simply connected
algebraic group $G$ is determined, up to an isomorphism, by its Tits index and the isomorphism
class of its anisotropic kernel $G_{an}$, subject to certain conditions on the Tits algebras,
together with an isomorphism $\Dyn(G_{an})\simeq\tf_{\min}$. Thus the classification consists in listing all
possible Tits indices of simple algebraic groups, and, for any given index, the conditions on the corresponding
anisotropic kernels. The necessary combinatorial restriction on a Tits index stated in Lemma~\ref{lem:opinv}
reduces possibilities to those listed in Proposition~\ref{prop:listcomb}. For some of them conditions on
the Tits algebras lead to a contradiction; for the rest they give criteria that anisotropic kernels must satisfy.


We represent Tits indices
graphically by Dynkin diagrams $D$ with the vertices in $J$ being circled; nontrivial $*$-action
is indicated by arrows $\longleftrightarrow$.
We also use the Tits notation ${}^mX_{n,r}^k$
for the groups of specific indices (see~\cite{Tits66}).
\medskip

We begin with simple groups of type $A_n$.
The split simple simply connected group of type $A_n$ over $R$ is $\SL_{n+1}(R)$;
the corresponding adjoint group is $\PGL_{n+1}(R)=\Aut(\MM_{n+1}(R))$. So the oriented simple simply connected
groups of inner type $A_n$ are of the form $\SL_1(A)$, where $A$ is an Azumaya algebra
over $R$ of degree $n+1$, uniquely determined up to an isomorphism. Obviously $A$ is the Tits algebra of
$\SL_1(A)$ corresponding to the natural representation of $\SL_{n+1}(R)$ in $R^{n+1}$; so
$[A]=\beta_{\SL_1(A)}(\omega_1)$. The change of orientation corresponds to the replacement of $A$ with $A^{op}$.
Note that $\SL_1(A)\simeq\SL_1(A^{op})$ as groups, the isomorphism being $g\mapsto g^{-1}$.

\begin{lem}\label{lem:Div}
Assume that $\SL_1(E)$ and $\SL_1(E')$ are anisotropic, and $[E]=[E']$ in $\Br(R)$. Then $E\simeq E'$.
\end{lem}
\begin{proof}
Since projective modules over $R$ are free, $[E]=[E']$ means that $\MM_n(E)\simeq\MM_m(E')$ for some
$n$ and $m$. Then $\SL_n(E)$ and $\SL_m(E')$ are isomorphic as oriented groups. Now $\SL_1(E)^n$ and
$\SL_1(E')^m$ are both anisotropic kernels of $G$, so they are isomorphic. In particular, they
have the same type, that is $m=n$, and the degrees of $E$ and $E'$ are equal. Theorem~\ref{thm:levi}~\eqref{item:bij}
implies that $\SL_1(E)^m$ and $\SL_1(E')^m$ are isomorphic as oriented groups, hence $\SL_1(E)$ and $\SL_1(E')$
are isomorphic as oriented groups, that is $E\simeq E'$.
\end{proof}\medskip

\begin{thm}[$\mathbf{{}^1A_n}$]\label{thm:An}
Every simple simply connected group $G$ of inner type $A_n$ over $R$ has the Tits index $({}^1A_n,\,J)$, where
$J=\{d,\,2d,\,\ldots,\,rd\}$, $n+1=(r+1)d$:
\begin{equation}\tag{${}^1A_{n,\,r}^{(d)}$}
\xymatrix@C=10pt@R=5pt{
\bullet\ar@{.}[r]&\bullet\ar@{-}[r]&*+[o][F]{\bullet}\ar@{-}[r]&\bullet\ar@{.}[r]&\bullet\ar@{-}[r]&*+[o][F]{\bullet}\ar@{-}[r]&\bullet
\ar@{.}[r]&\bullet\ar@{-}[r]&*+[o][F]{\bullet}\ar@{-}[r]&\bullet\ar@{.}[r]&\bullet\\
&&d&&&2d&&&rd
}
\end{equation}

$G$ is isomorphic to $\SL_{r+1}(E)$, and the anisotropic kernel is $\SL_1(E)^r$, where $\deg E=d$.

\end{thm}
\begin{proof}
Let $({}^1A_n,\,J)$ be the Tits index of $G$; we have $J=\{d,\,2d,\,\ldots,\,rd\}$ for some $d$ with $n+1=(r+1)d$ by
Lemma~\ref{lem:opinv} and Proposition~\ref{prop:listcomb}. The anisotropic kernel $G_{an}$ is isomorphic to $\SL_1(E_1)\times\ldots\SL_1(E_{r+1})$
for some Azumaya algebras $E_1$, \ldots, $E_r$. The Cartan matrix of $A_n$ shows that
$\alpha_{i\cdot d}=2\omega_{i\cdot d}-\omega_{i\cdot d-1}-\omega_{i\cdot d+1}$ for $i=1,\,\ldots,\,r$. By Theorem~\ref{thm:levi}, we have
$$
0=\beta_{G_{an}}(\alpha'_{i\cdot d})=\beta_{\SL_1(E_i)}(\omega_1)-\beta_{\SL_1(E_{i+1})}(\omega_1)=[E_i]-[E_{i+1}].
$$
Now Lemma~\ref{lem:Div} implies that all $E_i$ are isomorphic. Set $E=E_1$; then $\SL_{r+1}(E)$ has the same
Tits index and the same anisotropic kernel as $G$, so by Theorem~\ref{thm:levi} we have $G\simeq\SL_{r+1}(E)$, as claimed.
\end{proof}\medskip

The above result implies that for any Azumaya algebra $A$ over $R$, the group
$G=\SL_1(A)$ is isomorphic to $\SL_{r+1}(E)$, where $E$ is an Azumaya algebra such that $\SL_1(E)$
is anisotropic. In this case the degree of $E$ is called the \emph{index} of $A$ and is denoted by $\ind A$; obviously $\ind A$ divides
$\deg A$. The \emph{exponent} $\exp A$ of $A$ is the order of $[A]$ in $\Br(R)$. We will need the following result:

\begin{prop}\label{prop:exp} Let $A$ be an Azumaya algebra. Then $\exp A$ divides $\ind A$, and they have the same prime factors.
\end{prop}
\begin{proof}
The first part follows from the fact that $[A]=[E]=\beta_{\SL_1(E)}(\omega_1)$, and
$(\deg E)\omega_1$ belongs to the root lattice of $\SL_1(A)$.
The second part follows from \cite[Ch.~II, Thm.~1]{Ga}.
\end{proof}\medskip

Let $R'/R$ be an \'etale extension of degree $n$. We can interpret the corestriction homomorphism $\cores_{R'/R}\colon\Br(R')\to\Br(R)$
as follows. If $A$ is an Azumaya algebra over $R'$ of degree $d$, $R_{R'/R}(\SL_1(A))$ is a group of type $nA_{d-1}$ over $R$, with the
$*$-action permuting the copies of $A_{d-1}$. Now $\cores_{R'/R}([A])=\beta_{R_{R'/R}(\SL_1(A))}(\omega)$, where $\omega$ is the sum of
the fundamental weights $\omega_1$ of all copies of $A_{d-1}$ (cf. \cite[\S~5.3]{Tits71}).

\setcounter{thm}{2}
\begin{thm}[$\mathbf{{}^2A_n}$]\label{thm:An2}
Every simple simply connected group $G$ of type ${}^2A_n$ over $R$ has the Tits index $({}^2A_n,\,J)$,
where $J=\{d,\,2d,\,\ldots,\,rd,\,n+1-rd,\,\ldots,\,n+1-2d,\,n+1-d\}$ for some $r\ge 0$, $d>0$ such that
$d\mid n+1$, $2rd\le n+1$:

\medskip

\begin{equation}\tag{${}^2A_{n,\,r}^{(d)}$}
\xymatrix@C=6pt@R=5pt{
\bullet\ar@{.}[r]&\bullet\ar@{-}[r]&*+[o][F]{\bullet}\ar@{<->}@/^1.5pc/[rrrrrrrrrrr]\ar@{-}[r]&\bullet\ar@{.}[r]&\bullet\ar@{-}[r]&*+[o][F]{\bullet}
\ar@{<->}@/^1pc/[rrrrr]\ar@{-}[r]&\bullet
\ar@{.}[r]&\bullet\ar@{-}[r]&\bullet\ar@{.}[r]&\bullet\ar@{-}[r]&*+[o][F]{\bullet}\ar@{-}[r]&\bullet
\ar@{.}[r]&\bullet\ar@{-}[r]&*+[o][F]{\bullet}\ar@{-}[r]&\bullet\ar@{.}[r]&\bullet\\
&&{\quad d\quad }&&&{\quad rd\quad }&&&&&n+1-rd&&&n+1-d
}
\end{equation}

Denote by $\Spec R'$ the orbit corresponding to $\{1,\,n\}$ (so that $R'/R$ is a connected quadratic \'etale extension). The possible anisotropic
kernels are the following:
\begin{itemize}
\item $H\times R_{R'/R}(\SL_1(E))^r$, where $E$ is an Azumaya algebra over $R'$ with $\ind E=\deg E=d$, $H$ is a simple simply connected anisotropic of
type ${}^2A_{n-2rd}$ over $R$ whose orbit $O$ corresponding to $\{1,\,n-2rd\}$ is isomorphic to $\Spec R'$, such that $\beta_{H_O}(\omega_1)=[E]$,
when $n-2rd\ge 2$;
\item $\SL_1(A)\times R_{R'/R}(\SL_1(A_{R'}))^r$, where $A$ is an Azumaya algebras $A$ over $R$ such that $\ind A=\deg A=2$ and $\ind A_{R'}=d$,
when $n-2rd=1$;
\item $R_{R'/R}(\SL_1(E))^r$, where $E$ is an Azumaya algebra over $R$ such that $\ind E=\deg E=d$ and $\cores_{R'/R}([E])=0$, when $n-2rd\le 0$.
\end{itemize}
\end{thm}
\begin{proof}
Let $({}^2A_n,\,J)$ be the Tits index of $G$; we have $J=\{d,\,2d,\,\ldots,\,rd,\,n+1-rd,\,\ldots,\,n+1-2d,\,n+1-d\}$ for
some $r\ge 0$, $d>0$ with $d\mid n+1$, $2rd\le n+1$ by Lemma~\ref{lem:opinv} and Proposition~\ref{prop:listcomb}.
The anisotropic kernel $G_{an}$ is isomorphic to $H_1\times\ldots\times H_r\times H$, where $H_i$ are groups
of outer type $A_{d-1}+A_{d-1}$ with the $*$-action permuting two summands, and $H$ is a group of outer type
${}^2A_{n-2rd}$ when $n-2rd\ge 2$, is isomorphic to $\SL_1(A)$ for some Azumaya algebra $A$ over $R$ with $\ind A=\deg A=2$ when $n-2rd=1$,
and is trivial otherwise. Over $R'$ every $H_i$ becomes inner, hence we have $H_i\simeq R_{R'/R}(\SL_1(E_i))$
for some Azumaya algebra $E_i$ over $R'$, $\ind E_i=\deg E_i=d$.

Denote the orbit corresponding to $\{i\cdot d,\,n+1-i\cdot d\}$ by $O_i$, $i=1,\,\ldots,\,r$.
The Cartan matrix of $A_n$ shows that $\alpha_{i\cdot d}=2\omega_{i\cdot d}-\omega_{i\cdot d-1}-\omega_{i\cdot d+1}$. When $i<r$, by Theorem~\ref{thm:levi}
we have
$$
0=\beta_{{(G_{an})}_{O_i}}(\alpha'_{O_i})=\beta_{\SL_1(E_i)}(\omega_1)-\beta_{\SL_1(E_{i+1})}(\omega_1)=[E_i]-[E_{i+1}].
$$
Lemma~\ref{lem:Div} implies now that all $E_i$ are isomorphic; we set $E=E_1$.

In the case $n-2rd\ge 2$ by Theorem~\ref{thm:levi} we have
$$
0=\beta_{{(G_{an})}_{O_r}}(\alpha'_{O_r})=\beta_{\SL_1(E)}(\omega_1)-\beta_{H_{O_r}}(\omega_1)=[E]-\beta_{H_{O_r}}(\omega_1).
$$

In the case $n-2rd=1$ we have
$$
0=\beta_{{(G_{an})}_{O_r}}(\alpha'_{O_r})=\beta_{\SL_1(E)}(\omega_1)-\beta_{\SL_1(A)_{O_r}}(\omega_1)=[E]-[A_{R'}],
$$
for $O_r\simeq\Spec R'$ as a scheme.

In the case $n-2rd=0$ we have
$$
0=\beta_{{(G_{an})}_{O_r}}(\alpha'_{O_r})=\beta_{\SL_1(E)}(\omega_1)=[E],
$$
hence $E\simeq R'$. $G$ is quasi-split in this case.

Finally, in the case $n-2rd=-1$ we have $O_r\simeq\Spec R$, and hence
$$
0=\beta_{G_{an}}(\alpha'_{O_r})=\cores_{R'/R}(\beta_{\SL_1(E)}(\omega_1))=\cores_{R'/R}([E]).
$$
\end{proof}\medskip

\setcounter{thm}{2}
\begin{thm}[$\mathbf{B_n}$]\label{thm:Bn}
Every simple simply connected group of type $B_n$ over $R$, $n\ge 2$, has the Tits index $(B_n,\,J)$, where
$J=\{1,\,2,\,\ldots,\,r\}$ for some $r\ge 0$:

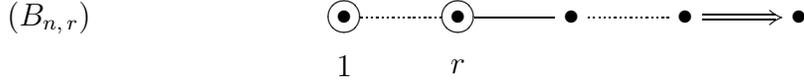
\begin{equation}\tag{$B_{n,\,r}$}
\xymatrix@C=30pt@R=5pt{
*+[o][F]{\bullet}\ar@{.}[r]&*+[o][F]{\bullet}\ar@{-}[r]&\bullet
\ar@{.}[r]&\bullet\ar@{=>}[r]&\bullet\\
1&r
}
\end{equation}

The possible anisotropic kernels are the following:
\begin{itemize}
\item simple simply connected anisotropic groups of type $B_{n-r}$ over $R$, when $n-r\ge 2$;
\item $\SL_1(A)$, where $A$ is an Azumaya algebras $A$ over $R$ with $\ind A=\deg A=2$, when $n-r=1$.
\end{itemize}
If $n=r$ then $G$ is split.
\end{thm}
\begin{proof}
Let $(B_n,\,J)$ be the Tits index of $G$ ; we have $J=\{d,\,2d,\,\ldots,\,rd\}$ for some $r\ge 0$, $d>0$ with $rd\le n$ by
Lemma~\ref{lem:opinv} and Proposition~\ref{prop:listcomb}. The anisotropic kernel $G_{an}$ is isomorphic to
$\SL_1(E_1)\times\ldots\SL_1(E_r)\times H$, where $H$ is a group of type $B_{n-rd}$ when $n-rd\ge 2$,
is isomorphic to $\SL_1(A)$ for some Azumaya algebra $A$ over $R$ with $\ind A=\deg A=2$ when $n-rd=1$, or is trivial when $n=rd$.

In the case $n-rd\ge 2$ the Cartan matrix of $B_n$ shows that
$\alpha_{rd}=2\omega_{rd}-\omega_{rd-1}-\omega_{rd+1}$. By Theorem~\ref{thm:levi}, we have
$$
0=\beta_{G_{an}}(\alpha'_{rd})=\beta_{\SL_1(E_r)}(\omega_1)-\beta_{H}(\omega_1)=[E_r].
$$
So $E_r=R$, hence $d=1$.

In the case $n-rd=1$ we have $\alpha_{rd}=2\omega_{n-1}-\omega_{n-2}-2\omega_n$, so
$$
0=\beta_{G_{an}}(\alpha'_{rd})=\beta_{\SL_1(E_r)}(\omega_1)-2\beta_H(\omega_1)=[E_r],
$$
and again $d=1$.

Finally, in the case $n=rd$ we have $\alpha_{rd}=2\omega_n-\omega_{n-1}$, so
$$
0=\beta_{G_{an}}(\alpha'_{rd})=\beta_{\SL_1(E_r)}(\omega_1)=[E_r],
$$
$d=1$, and $G$ is split in this case.
\end{proof}\medskip

The split simple simply connected group scheme of type $C_n$ over $R$ is $\Sp_{2n}(R)$.

\begin{prop}\label{prop:sp}
Assume that $G$ is a simple simply connected group of type $C_n$ over $R$, $\beta_G(\omega_1)=[E]$, $\ind E=d$.
Then $d=2^k$ for some $k\ge 0$ and $d\mid 2n$. If $d=1$ then $G$ is split.
\end{prop}
\begin{proof}
We have $2[E]=0$, since $2\omega_1$ belongs to $\Lar$. Now Proposition~\ref{prop:exp} implies that $d=2^k$.

The vector representation $\rho\colon\Sp_{2n}(R)\to\GL(R^{2n})$ is center preserving and has a weight
$\omega_1$; so $[A_\rho]=[E]$. But $A_\rho$ has degree $2n$, so $d\mid 2n$.

If $d=1$ then by Proposition~\ref{prop:strin} $G$ corresponds to an element of $\HH^1(R,\,\Sp_{2n})$, and the latter is trivial by
\cite[Ch.~I, Cor.~4.1.2]{Knus}.
\end{proof}\medskip

\setcounter{thm}{2}\begin{thm}[$\mathbf{C_n}$]\label{thm:Cn}
Every simple simply connected group $G$ of type $C_n$ over $R$, $n\ge 2$, has the Tits index
$(C_n,\,J)$, where
$J=\{d,\,2d,\,\ldots,\,rd\}$ for some $r\ge 0$, $d>0$ such that $d=2^k\mid 2n$, $rd\le n$, and $r=n$ when
$d=1$:

\begin{equation}\tag{$C_{n,\,r}^{(d)}$}
\xymatrix@C=20pt@R=5pt{
\bullet\ar@{.}[r]&\bullet\ar@{-}[r]&*+[o][F]{\bullet}\ar@{-}[r]&\bullet\ar@{.}[r]&\bullet\ar@{-}[r]&*+[o][F]{\bullet}\ar@{-}[r]&\bullet
\ar@{.}[r]&\bullet\ar@{<=}[r]&\bullet\\
&&d&&&rd
}
\end{equation}

The possible anisotropic kernels are the following:
\begin{itemize}
\item $H\times\SL_1(E)^r$, where $E$ is an Azumaya algebra over $R$ with $\ind E=\deg E=d$, $H$ is a simple simply connected anisotropic of type
$C_{n-rd}$ over $R$ with $\beta_H(\omega_1)=[E]$, when $n-rd\ge 2$;
\item $\SL_1(E)^{r+1}$, where $E$ is an Azumaya algebras $E$ over $R$ with $\ind E=\deg E=d$, when $n-rd=1$;
\item $\SL_1(E)^r$, where $E$ is an Azumaya algebras $E$ over $R$ with $\ind E=\deg E=d$ and $\exp E\le 2$, when $n-rd=0$.
\end{itemize}
\end{thm}
\begin{proof}
Let $(C_n,\,J)$ be the Tits index of $G$; we have $J=\{d,\,2d,\,\ldots,\,rd\}$ for some $r\ge 0$, $d>0$ with $rd\le n$ by
Lemma~\ref{lem:opinv} and Proposition~\ref{prop:listcomb}. The anisotropic kernel $G_{an}$ is isomorphic to
$\SL_1(E_1)\times\ldots\SL_1(E_r)\times H$, where $H$ is a group of type $C_{n-rd}$ when $n-rd\ge 2$,
is isomorphic to $\SL_1(A)$ for some Azumaya algebra $A$ over $R$ with $\ind A=\deg A=2$ when $n-rd=1$, or is trivial when $n=rd$.

The Cartan matrix of $C_n$ shows that
$\alpha_{i\cdot d}=2\omega_{i\cdot d}-\omega_{i\cdot d-1}-\omega_{i\cdot d+1}$ for $i=1,\,\ldots,\,r-1$. By Theorem~\ref{thm:levi}, we have
$$
0=\beta_{G_{an}}(\alpha'_{i\cdot d})=\beta_{\SL_1(E_i)}(\omega_1)-\beta_{\SL_1(E_{i+1})}(\omega_1)=[E_i]-[E_{i+1}].
$$
Lemma~\ref{lem:Div} implies now that all $E_i$ are isomorphic; set $E=E_1$. Note that $[E]=\beta_G(\omega_1)$, hence by Proposition~\ref{prop:sp}
$d=2^k\mid 2n$, and $G$ is split when $d=1$.

In the case $n-rd\ge 2$ the Cartan matrix of $C_n$ shows that $\alpha_{rd}=2\omega_{rd}-\omega_{rd-1}-\omega_{rd+1}$, so
$$
0=\beta_{G_{an}}(\alpha'_{rd})=\beta_{\SL_1(E)}(\omega_1)-\beta_H(\omega_1)=[E]-\beta_H(\omega_1).
$$

In the case $n-rd=1$ we have $\alpha_{rd}=2\omega_{n-1}-\omega_{n-2}-\omega_n$, so
$$
0=\beta_{G_{an}}(\alpha'_{rd})=\beta_{\SL_1(E)}(\omega_1)-\beta_{\SL_1(A)}(\omega_1)=[E]-[A].
$$
Hence $[E]=[A]$ and $d=2$.

Finally, in the case $n=rd$ we have $\alpha_{rd}=2\omega_n-2\omega_{n-1}$, so
$$
0=\beta_{G_{an}}(\alpha'_{rd})=2\beta_{\SL_1(E)}(\omega_1)=2[E],
$$
that is $\exp E\le 2$.
\end{proof}\medskip

The split simple simply connected group scheme of type $D_n$ over $R$ is $\Spin_{2n}(R)$.

\begin{prop}\label{prop:so} Assume that $G$ is a simple simply connected group of type ${}^1D_n$ or ${}^2D_n$ over $R$, $n\ge 4$, $\beta_G(\omega_1)=[E]$,
$\ind E=d$. Then $d=2^k$ for some $k\ge 0$ and $d\mid 2n$.
\end{prop}
\begin{proof}
We have $2[E]=0$, since $2\omega_1$ belongs to $\Lar$. Now Proposition~\ref{prop:exp} implies that $d=2^k$.

The vector representation $\rho\colon\Spin_{2n}(R)\to\GL(R^{2n})$ is center preserving and has a weight $\omega_1$; so
$[A_\rho]=[E]$. But $A_\rho$ has degree $2n$, so $d\mid 2n$.
\end{proof}\medskip

\setcounter{thm}{2}\begin{thm}[$\mathbf{{}^1D_n}$]\label{thm:Dn}
Every simple simply connected group $G$ of inner type $D_n$ over $R$, $n\ge 4$, has the Tits index $({}^1D_n,\,J)$, where
$J=\{d,\,2d,\,\ldots,\,rd\}$ (possibly, after interchanging $n-1$ and $n$) for some $r\ge 0$, $d>0$
such that $d=2^k\mid 2n$, $rd\le n$, $n\ne rd+1$:

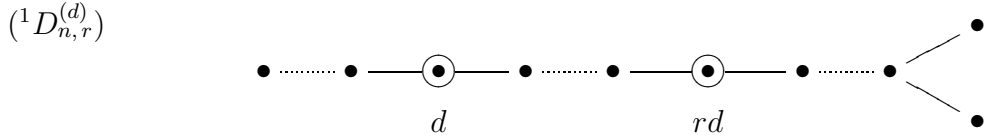
\begin{equation}\tag{${}^1D_{n,\,r}^{(d)}$}
\xymatrix@R=5pt@C=20pt{
&&&&&&&&\bullet\\
\bullet\ar@{.}[r]&\bullet\ar@{-}[r]&*+[o][F]{\bullet}\ar@{-}[r]&\bullet\ar@{.}[r]&\bullet\ar@{-}[r]&*+[o][F]{\bullet}\ar@{-}[r]&\bullet
\ar@{.}[r]&\bullet\ar@{-}[ru]\ar@{-}[rd]\\
&&d&&&rd&&&\bullet
}
\end{equation}

The possible anisotropic kernels are the following:
\begin{itemize}
\item $H\times\SL_1(E)^r$, where $E$ is an Azumaya algebra over $R$ with $\ind E=\deg E=d$, $H$ is a simple simply connected anisotropic group
of inner type $D_{n-rd}$ over $R$ with $\beta_H(\omega_1)=[E]$, when $n-rd\ge 4$;
\item $\SL_1(A)\times\SL_1(E)^r$, where $E$ is an Azumaya algebra over $R$ with $\ind E=\deg E=d$, $A$ is an Azumaya algebra over $R$ with
$\ind A=\deg A=4$ such that $2[A]=[E]$, when $n-rd=3$;
\item $\SL_1(A_1)\times\SL_1(A_2)\times\SL_1(E)^r$, where $E$ is an Azumaya algebra over $R$ with $\ind E=\deg E=d$, $A_1$ and $A_2$ are
Azumaya algebras over $R$ such that $\ind A_1=\deg A_1=\ind A_2=\deg A_2=2$ and $[A_1]+[A_2]=[E]$, when $n-rd=2$;
\item $\SL_1(E)^r$, where $E$ is an Azumaya algebra over $R$ with $\ind E=\deg E=d$ and $\exp E\le 2$, when $n=rd$.
\end{itemize}
\end{thm}
\begin{proof}
Let $({}^1D_n,\,J)$ be the Tits index of $G$; we have $J=\{d,\,2d,\,\ldots,\,rd\}$ for some $r\ge 0$, $d>0$ with $rd\le n$, $rd\ne n-1$ by
Lemma~\ref{lem:opinv} and Proposition~\ref{prop:listcomb}. The anisotropic kernel $G_{an}$ is isomorphic to
$\SL_1(E_1)\times\ldots\SL_1(E_r)\times H$, where $H$ is a group of inner type $D_{n-rd}$ when $n-rd\ge 4$,
is isomorphic to $\SL_1(A)$ for some Azumaya algebra $A$ over $R$ with $\ind A=\deg A=4$ when $n-rd=3$
is isomorphic to $\SL_1(A_1)\times\SL_1(A_2)$ for some Azumaya algebras $A_1$, $A_2$ over $R$ with $\ind A_1=\deg A_1=\ind A_2=\deg A_2=2$,
or is trivial when $n=rd$.

The Cartan matrix of $D_n$ shows that
$\alpha_{i\cdot d}=2\omega_{i\cdot d}-\omega_{i\cdot d-1}-\omega_{i\cdot d+1}$ for $i=1,\,\ldots,\,r-1$. By Theorem~\ref{thm:levi}, we have
$$
0=\beta_{G_{an}}(\alpha'_{i\cdot d})=\beta_{\SL_1(E_i)}(\omega_1)-\beta_{\SL_1(E_{i+1})}(\omega_1)=[E_i]-[E_{i+1}].
$$
Lemma~\ref{lem:Div} implies now that all $E_i$ are isomorphic; set $E=E_1$. Note that $[E]=\beta_G(\omega_1)$, hence by Proposition~\ref{prop:so}
$d=2^k\mid 2n$.

In the case $n-rd\ge 4$ the Cartan matrix of $D_n$ shows that $\alpha_{rd}=2\omega_{rd}-\omega_{rd-1}-\omega_{rd+1}$, so
$$
0=\beta_{G_{an}}(\alpha'_{rd})=\beta_{\SL_1(E)}(\omega_1)-\beta_H(\omega_1)=[E]-\beta_H(\omega_1).
$$

In the case $n-rd=3$ we still have $\alpha_{rd}=2\omega_{rd}-\omega_{rd-1}-\omega_{rd+1}$, so
$$
0=\beta_{G_{an}}(\alpha'_{rd})=\beta_{\SL_1(E)}(\omega_1)-\beta_{\SL_1(A)}(\omega_2)=[E]-2[A].
$$

In the case $n-rd=2$ we have $\alpha_{rd}=2\omega_{n-2}-\omega_{n-3}-\omega_{n-1}-\omega_n$, so
$$
0=\beta_{G_{an}}(\alpha'_{rd})=\beta_{\SL_1(E)}(\omega_1)-\beta_{\SL_1(A_1)}(\omega_1)-\beta_{\SL_1(A_1)}(\omega_2)=[E]-[A_1]-[A_2].
$$

Finally, in the case $n=rd$ we have $\alpha_{rd}=2\omega_n-\omega_{n-2}$, so
$$
0=\beta_{G_{an}}(\alpha'_{rd})=\beta_{\SL_1(E)}(\omega_2)=2[E],
$$
hence $\exp E\le 2$.
\end{proof}\medskip

\setcounter{thm}{2}\begin{thm}[$\mathbf{{}^2D_n}$]\label{thm:Dn2}
Every simple simply connected group $G$ of type ${}^2D_n$, $n\ge 4$, has the Tits index $({}^2D_n,\,J)$, where
$J=\{d,\,2d,\,\ldots,\,rd\}$ for some $r\ge 0$, $d>0$ such that $d=2^k\mid 2n$, $rd< n-1$,
or $J=\{d,\,2d,\,\ldots,\,(r-1)d,\,n-1,\,n\}$ for some $r\ge 0$, $d\in \{1,\,2\}$
such that $rd=n-1$.

\begin{equation}\tag{${}^2D_{n,\,r}^{(d)}$}
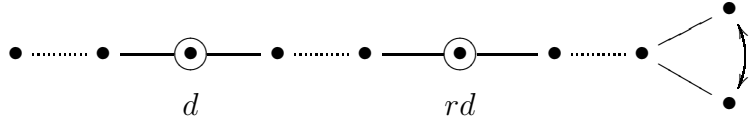
\xymatrix@R=5pt@C=20pt{
&&&&&&&&\bullet\ar@{<->}@/^/[dd]\\
\bullet\ar@{.}[r]&\bullet\ar@{-}[r]&*+[o][F]{\bullet}\ar@{-}[r]&\bullet\ar@{.}[r]&\bullet\ar@{-}[r]&*+[o][F]{\bullet}\ar@{-}[r]&\bullet
\ar@{.}[r]&\bullet\ar@{-}[ru]\ar@{-}[rd]\\
&&d&&&rd&&&\bullet
}
\end{equation}

Denote by $\Spec R'$ the orbit corresponding to $\{n-1,\,n\}$ (so that $R'/R$ is a connected quadratic \'etale extension). The possible anisotropic
kernels are the following:
\begin{itemize}
\item $H\times\SL_1(E)^r$, where $E$ is an Azumaya algebra over $R$ with $\ind E=\deg E=d$, $H$ is a simple simply connected anisotropic group
of type ${}^2D_{n-rd}$ over $R$ whose orbit corresponding to $\{n-rd-1,\,n-rd\}$ is isomorphic to $\Spec R'$, such that $\beta_H(\omega_1)=[E]$,
when $n-rd\ge 4$;
\item $H\times\SL_1(E)^r$, where $E$ is an Azumaya algebra over $R$ with $\ind E=\deg E=d$, $H$ is a simple simply connected anisotropic group
of type ${}^2A_3$ over $R$ whose orbit corresponding to $\{1,\,3\}$ is isomorphic to $\Spec R'$, such that $\beta_H(\omega_2)=[E]$, when $n-rd=3$;
\item $R_{R'/R}(\SL_1(A))\times\SL_1(E)^r$, where $E$ is an Azumaya algebra over $R$ with $\ind E=\deg E=d$, $A$ is an Azumaya algebras over $R'$ such that
$\ind A=\deg A=2$ and $\cores_{R'/R}([A])=[E]$, when $n-rd=2$;
\item $\SL_1(E)^r$, where $E$ is an Azumaya algebra over $R$ such that $\ind E=\deg E=d$ and $[E_{R'}]=0$, when $n-rd=1$.
\end{itemize}
\end{thm}
\begin{proof}
Let $({}^2D_n,\,J)$ be the Tits index of $G$; by Lemma~\ref{lem:opinv} and Proposition~\ref{prop:listcomb} we have $J=\{d,\,2d,\,\ldots,\,rd\}$
(or $J=\{d,2d,\ldots,(r-2)d,n-1,n\}$ in the case $rd=n-1$) for some $r\ge 0$, $d>0$ with $rd\ne n-1$. The anisotropic kernel $G_{an}$ is isomorphic to
$\SL_1(E_1)\times\ldots\SL_1(E_r)\times H$, where $H$ is a group of type ${}^2D_{n-rd}$ when $n-rd\ge 4$,
of type ${}^2A_3$ when $n-rd=3$, is isomorphic to $R_{R'/R}(\SL_1(A))$ for some Azumaya algebras $A$ over a connected quadratic \'etale extension
$R'/R$ with $\ind A=\deg A=2$, or is trivial when $n-rd=1$.

Denote the orbit corresponding to $\{i\cdot d,\,n+1-i\cdot d\}$ by $O_i$, $i=1,\,\ldots,\,r$. The Cartan matrix of $D_n$ shows that
$\alpha_{i\cdot d}=2\omega_{i\cdot d}-\omega_{i\cdot d-1}-\omega_{i\cdot d+1}$ for $i=1,\,\ldots,\,r-1$. By Theorem~\ref{thm:levi}, we have
$$
0=\beta_{{(G_{an})}_{O_i}}(\alpha'_{O_i})=\beta_{\SL_1(E_i)}(\omega_1)-\beta_{\SL_1(E_{i+1})}(\omega_1)=[E_i]-[E_{i+1}].
$$
Lemma~\ref{lem:Div} implies now that all $E_i$ are isomorphic; set $E=E_1$. Note that $[E]=\beta_G(\omega_1)$, hence by Proposition~\ref{prop:so}
$d=2^k\mid 2n$.

In the case $n-rd\ge 4$ the Cartan matrix of $D_n$ shows that $\alpha_{rd}=2\omega_{rd}-\omega_{rd-1}-\omega_{rd+1}$, so
$$
0=\beta_{{(G_{an})}_{O_r}}(\alpha'_{O_r})=\beta_{\SL_1(E)}(\omega_1)-\beta_H(\omega_1)=[E]-\beta_H(\omega_1).
$$

In the case $n-rd=3$ we still have $\alpha_{rd}=2\omega_{rd}-\omega_{rd-1}-\omega_{rd+1}$, so
$$
0=\beta_{{(G_{an})}_{O_r}}(\alpha'_{O_r})=\beta_{\SL_1(E)}(\omega_1)-\beta_H(\omega_2).
$$

In the case $n-rd=2$ $O_r\simeq\Spec R$, and we have $\alpha_{rd}=2\omega_{n-2}-\omega_{n-3}-\omega_{n-1}-\omega_n$, so
$$
0=\beta_{{(G_{an})}_{O_r}}(\alpha'_{O_r})=\beta_{\SL_1(E)}(\omega_1)-\cores_{R'/R}(\beta_{\SL_1(A)}(\omega_1))=[E]-\cores_{R'/R}([A]).
$$

Finally, in the case $n-rd=1$ the condition $d|2n$ implies $d\in\{1,2\}$; also, $O_r\simeq\Spec R'$, and we have $\alpha_{rd}=2\omega_n-\omega_{n-2}$, so
$$
0=\beta_{{(G_{an})}_{O_r}}(\alpha'_{O_r})=\beta_{\SL_1(E)_{O_r}}(\omega_1)=[E_{R'}].
$$
\end{proof}\medskip

\setcounter{thm}{2}\begin{thm}[$\mathbf{{}^3D_4}$ and $\mathbf{{}^6D_4}$]\label{thm:D43}
Every simple simply connected group $G$ of type ${}^3D_4$ or ${}^6D_4$ over $R$ has one of the following Tits indices:

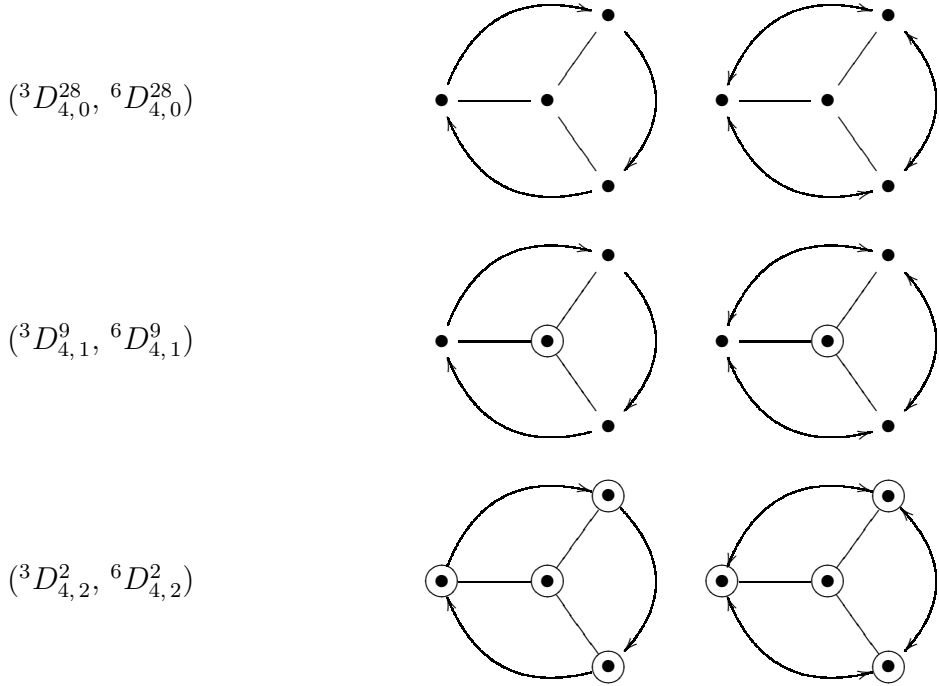
\begin{equation}\tag{${}^3D_{4,\,0}^{28}$, ${}^6D_{4,\,0}^{28}$}
\begin{array}{cccc}
\xymatrix@C=10pt@R=20pt{
&&&\bullet\ar@{->}@/^1.5pc/[dd]\\
\bullet\ar@{-}[rr]\ar@{->}@/^1.5pc/[rrru]\ar@{<-}@/_1.5pc/[rrrd]&&\bullet\ar@{-}[ru]\ar@{-}[rd]\\
&&&\bullet
}
&&&
\xymatrix@C=10pt@R=20pt{
&&&\bullet\ar@{<->}@/^1.5pc/[dd]\\
\bullet\ar@{-}[rr]\ar@{<->}@/^1.5pc/[rrru]\ar@{<->}@/_1.5pc/[rrrd]&&\bullet\ar@{-}[ru]\ar@{-}[rd]\\
&&&\bullet
}
\end{array}
\end{equation}

\begin{equation}\tag{${}^3D_{4,\,1}^{9}$, ${}^6D_{4,\,1}^{9}$}
\begin{array}{cccc}
\xymatrix@C=10pt@R=20pt{
&&&\bullet\ar@{->}@/^1.5pc/[dd]\\
\bullet\ar@{-}[rr]\ar@{->}@/^1.5pc/[rrru]\ar@{<-}@/_1.5pc/[rrrd]&&*+[o][F]{\bullet}\ar@{-}[ru]\ar@{-}[rd]\\
&&&\bullet
}
&&&
\xymatrix@C=10pt@R=20pt{
&&&\bullet\ar@{<->}@/^1.5pc/[dd]\\
\bullet\ar@{-}[rr]\ar@{<->}@/^1.5pc/[rrru]\ar@{<->}@/_1.5pc/[rrrd]&&*+[o][F]{\bullet}\ar@{-}[ru]\ar@{-}[rd]\\
&&&\bullet
}
\end{array}
\end{equation}

\begin{equation}\tag{${}^3D_{4,\,2}^{2}$, ${}^6D_{4,\,2}^{2}$}
\begin{array}{cccc}
\xymatrix@C=10pt@R=20pt{
&&&*+[o][F]{\bullet}\ar@{->}@/^1.5pc/[dd]\\
*+[o][F]{\bullet}\ar@{-}[rr]\ar@{->}@/^1.5pc/[rrru]\ar@{<-}@/_1.5pc/[rrrd]&&*+[o][F]{\bullet}\ar@{-}[ru]\ar@{-}[rd]\\
&&&*+[o][F]{\bullet}
}
&&&
\xymatrix@C=10pt@R=20pt{
&&&*+[o][F]{\bullet}\ar@{<->}@/^1.5pc/[dd]\\
*+[o][F]{\bullet}\ar@{-}[rr]\ar@{<->}@/^1.5pc/[rrru]\ar@{<->}@/_1.5pc/[rrrd]&&*+[o][F]{\bullet}\ar@{-}[ru]\ar@{-}[rd]\\
&&&*+[o][F]{\bullet}
}
\end{array}
\end{equation}

Denote by $\Spec R'$ the orbit corresponding to $\{1,\,3,\,4\}$ (so that $R'/R$ is a connected cubic \'etale extension).
The possible anisotropic kernels in the case of ${}^3D_{4,\,1}^{9}$ or ${}^6D_{4,\,1}^{9}$ are of the form
$R_{R'/R}(\SL_1(A))$, where $A$ is an Azumaya algebra over $R'$ with
$\ind A=\deg A=2$ and $\cores_{R'/R}([A])=0$.

In the case of ${}^3D_{4,\,0}^{28}$ or ${}^6D_{4,\,0}^{28}$ $G$ is anisotropic; in the case of ${}^3D_{4,\,2}^{2}$ or ${}^6D_{4,\,2}^{2}$ $G$ is quasi-split.
\end{thm}
\begin{proof}
By Lemma~\ref{lem:opinv} and Proposition~\ref{prop:listcomb} $G$ is anisotropic or quasi-split, or has the Tits index
${}^3D_{4,\,1}^{9}$ or ${}^6D_{4,\,1}^{9}$.

Let the Tits index be ${}^3D_{4,\,1}^{9}$ or ${}^6D_{4,\,1}^{9}$. The anisotropic kernel $G_{an}$ is isomorphic to
$R_{R'/R}(\SL_1(A))$ for some Azumaya algebra over $R'$ with $\ind A=\deg A=2$.
The Cartan matrix of $D_4$ shows that $\alpha_2=2\omega_2-\omega_1-\omega_3-\omega_4$, so by Theorem~\ref{thm:levi} we have
$$
0=\beta_{G_{an}}(\alpha'_2)=-\cores_{R'/R}(\beta_{\SL_1(A)}(\omega_1))=-\cores_{R'/R}([A]).
$$
\end{proof}\medskip

\setcounter{thm}{2}\begin{thm}[$\mathbf{{}^1E_6}$]\label{thm:E6}
Every simple simply connected group $G$ of inner type $E_6$ over $R$ has one of the following Tits indices:

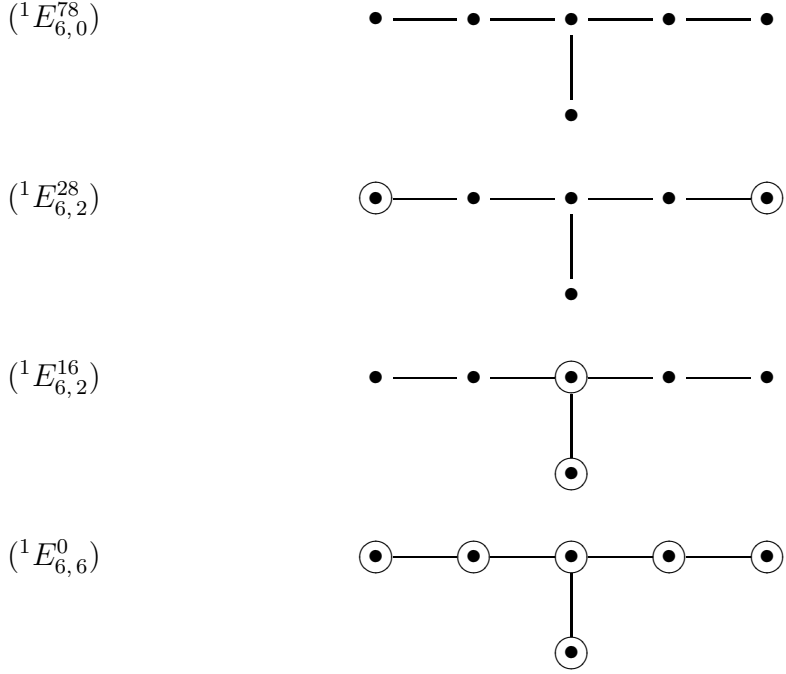
\begin{equation}\tag{${}^1E_{6,\,0}^{78}$}
\xymatrix{
{\bullet}\ar@{-}[r]&{\bullet}\ar@{-}[r]&{\bullet}\ar@{-}[d]\ar@{-}[r]&{\bullet}\ar@{-}[r]&{\bullet}\\
&&{\bullet}
}
\end{equation}

\begin{equation}\tag{${}^1E_{6,\,2}^{28}$}
\xymatrix{
*+[o][F]{\bullet}\ar@{-}[r]&{\bullet}\ar@{-}[r]&{\bullet}\ar@{-}[d]\ar@{-}[r]&{\bullet}\ar@{-}[r]&*+[o][F]{\bullet}\\
&&{\bullet}
}
\end{equation}

\begin{equation}\tag{${}^1E_{6,\,2}^{16}$}
\xymatrix{
{\bullet}\ar@{-}[r]&{\bullet}\ar@{-}[r]&*+[o][F]{\bullet}\ar@{-}[d]\ar@{-}[r]&{\bullet}\ar@{-}[r]&{\bullet}\\
&&*+[o][F]{\bullet}
}
\end{equation}

\begin{equation}\tag{${}^1E_{6,\,6}^{0}$}
\xymatrix{
*+[o][F]{\bullet}\ar@{-}[r]&*+[o][F]{\bullet}\ar@{-}[r]&*+[o][F]{\bullet}\ar@{-}[d]\ar@{-}[r]&*+[o][F]{\bullet}\ar@{-}[r]&*+[o][F]{\bullet}\\
&&*+[o][F]{\bullet}
}
\end{equation}

The possible anisotropic kernels are the following:
\begin{itemize}
\item simple simply connected anisotropic groups $H$ of type $D_4$ over $R$ with $\beta_H=0$, in the case of ${}^1E_{6,\,2}^{28}$;
\item $\SL_1(A)^2$, where $A$ is an Azumaya algebra over $R$ with $\ind A=\deg A=3$, in the case of ${}^1E_{6,\,2}^{16}$.
\end{itemize}

In the case of ${}^1E_{6,\,0}^{78}$ $G$ is anisotropic; in the case of ${}^1E_{6,\,6}^{0}$ $G$ is split.
\end{thm}
\begin{proof}
By Lemma~\ref{lem:opinv} and Proposition~\ref{prop:listcomb} the Tits index of $G$ is either one of the listed above or
the following:
$$
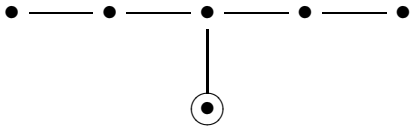
\xymatrix{
{\bullet}\ar@{-}[r]&{\bullet}\ar@{-}[r]&{\bullet}\ar@{-}[d]\ar@{-}[r]&{\bullet}\ar@{-}[r]&{\bullet}\\
&&*+[o][F]{\bullet}
}
$$

Let us first exclude the latter case. The anisotropic kernel $G_{an}$ is isomorphic to $\SL_1(A)$ for some
Azumaya algebra $A$ over $R$ with $\ind A=\deg A=6$. The Cartan matrix of $E_6$ shows that $\alpha_2=2\omega_2-\omega_4$.
By Theorem~\ref{thm:levi} we have
$$
0=\beta_{G_{an}}(\alpha'_2)=-\beta_{\SL_1(A)}(\omega_3)=-3[A].
$$
Hence $\exp A=3$, but this contradicts  Proposition~\ref{prop:exp}.

In the case of ${}^1E_{6,\,2}^{28}$ the anisotropic kernel $G_{an}$ is of type ${}^1D_4$. We have
$\alpha_1=2\omega_1-\omega_3$, $\alpha_6=2\omega_6-\omega_5$, so
\begin{align*}
&0=\beta_{G_{an}}(\alpha'_1)=-\beta_{G_{an}}(\omega_1);\\
&0=\beta_{G_{an}}(\alpha'_6)=-\beta_{G_{an}}(\omega_4).
\end{align*}
It follows that $\beta_{G_{an}}=0$.

In the case of ${}^1E_{6,\,2}^{16}$ the anisotropic kernel $G_{an}$ is isomorphic to $\SL_1(A_1)\times\SL_1(A_2)$
for some Azumaya algebras $A_1$, $A_2$ over $R$ with $\ind A_1=\deg A_1=\ind A_2=\deg A_2=3$. We have
$\alpha_4=2\omega_4-\omega_2-\omega_3-\omega_5$, so
$$
0=\beta_{G_{an}}(\alpha'_4)=\beta_{\SL_1(A_1)}(\omega_1)-\beta_{\SL_1(A_2)}(\omega_1)=[A_1]-[A_2].
$$
By Lemma~\ref{lem:Div} $A_1\simeq A_2$.
\end{proof}\medskip

\setcounter{thm}{2}\begin{thm}[$\mathbf{{}^2E_6}$]\label{thm:E62}
Every simple simply connected group $G$ of type ${}^2E_6$ over $R$ has one of the following Tits indices:

\medskip

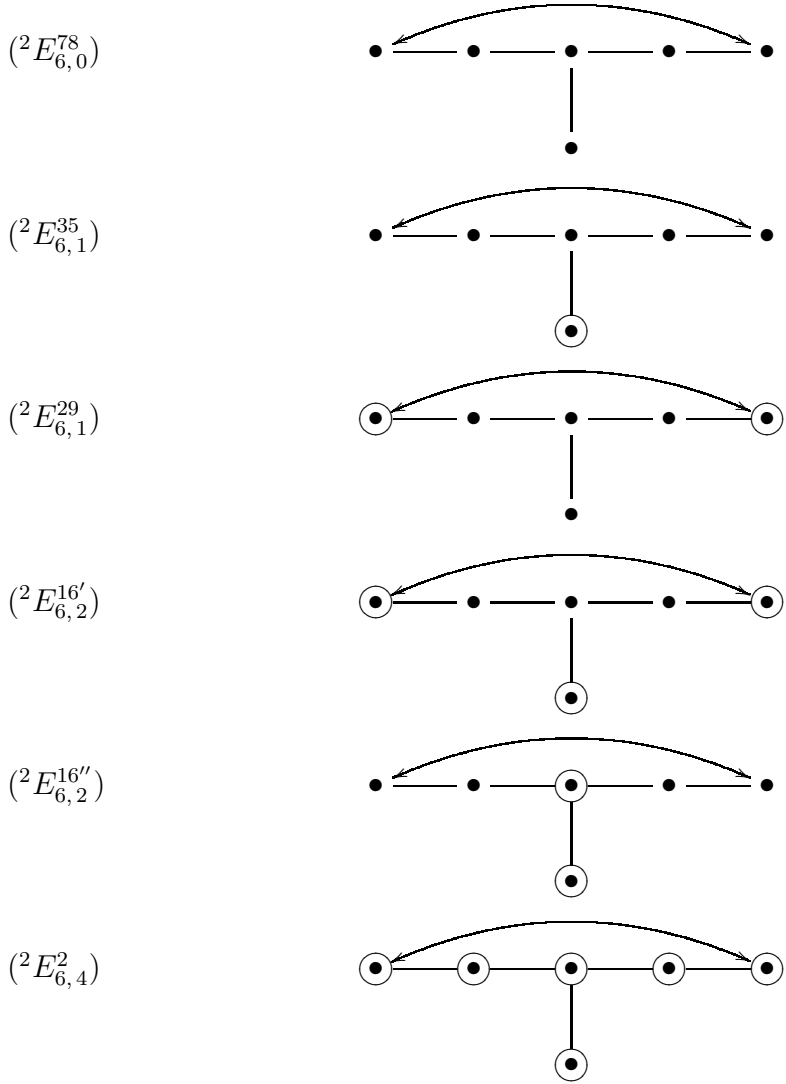
\begin{equation}\tag{${}^2E_{6,\,0}^{78}$}
\xymatrix{
{\bullet}\ar@{-}[r]\ar@{<->}@/^1.5pc/[rrrr]&{\bullet}\ar@{-}[r]&{\bullet}\ar@{-}[d]\ar@{-}[r]&{\bullet}\ar@{-}[r]&{\bullet}\\
&&{\bullet}
}
\end{equation}

\medskip

\begin{equation}\tag{${}^2E_{6,\,1}^{35}$}
\xymatrix{
{\bullet}\ar@{-}[r]\ar@{<->}@/^1.5pc/[rrrr]&{\bullet}\ar@{-}[r]&{\bullet}\ar@{-}[d]\ar@{-}[r]&{\bullet}\ar@{-}[r]&{\bullet}\\
&&*+[o][F]{\bullet}
}
\end{equation}

\medskip

\begin{equation}\tag{${}^2E_{6,\,1}^{29}$}
\xymatrix{
*+[o][F]{\bullet}\ar@{-}[r]\ar@{<->}@/^1.5pc/[rrrr]&{\bullet}\ar@{-}[r]&{\bullet}\ar@{-}[d]\ar@{-}[r]&{\bullet}\ar@{-}[r]&*+[o][F]{\bullet}\\
&&{\bullet}
}
\end{equation}

\medskip

\begin{equation}\tag{${}^2E_{6,\,2}^{16'}$}
\xymatrix{
*+[o][F]{\bullet}\ar@{-}[r]\ar@{<->}@/^1.5pc/[rrrr]&{\bullet}\ar@{-}[r]&{\bullet}\ar@{-}[d]\ar@{-}[r]&{\bullet}\ar@{-}[r]&*+[o][F]{\bullet}\\
&&*+[o][F]{\bullet}
}
\end{equation}

\medskip

\begin{equation}\tag{${}^2E_{6,\,2}^{16''}$}
\xymatrix{
{\bullet}\ar@{-}[r]\ar@{<->}@/^1.5pc/[rrrr]&{\bullet}\ar@{-}[r]&*+[o][F]{\bullet}\ar@{-}[d]\ar@{-}[r]&{\bullet}\ar@{-}[r]&{\bullet}\\
&&*+[o][F]{\bullet}
}
\end{equation}

\medskip

\begin{equation}\tag{${}^2E_{6,\,4}^{2}$}
\xymatrix{
*+[o][F]{\bullet}\ar@{-}[r]\ar@{<->}@/^1.5pc/[rrrr]&*+[o][F]{\bullet}\ar@{-}[r]&*+[o][F]{\bullet}\ar@{-}[d]\ar@{-}[r]
&*+[o][F]{\bullet}\ar@{-}[r]&*+[o][F]{\bullet}\\
&&*+[o][F]{\bullet}
}
\end{equation}

Denote by $\Spec R'$ the orbit corresponding to $\{1,\,6\}$ (so that $R'/R$ is a connected quadratic \'etale extension). The possible anisotropic
kernels are the following:
\begin{itemize}
\item simple simply connected anisotropic groups $H$ of type ${}^2A_5$ over $R$ with $\beta_H(\omega_3)=0$, in the case of ${}^2E_{6,\,1}^{35}$;
\item simple simply connected anisotropic groups $H$ of type ${}^2D_4$ over $R$ with $\beta_{H_{R'}}(\omega_3)=0$, in the case of ${}^2E_{6,\,1}^{29}$;
\item simple simply connected anisotropic groups $H$ of type ${}^2A_3$ over $R$ with $\beta_H(\omega_2)=0$ and $\beta_{H_{R'}}(\omega_1)=0$,
in the case of ${}^2E_{6,\,2}^{16'}$;
\item $R_{R'/R}(\SL_1(A))$, where $A$ is an Azumaya algebra over $R'$ with $\ind A=\deg A=3$ and $\cores_{R'/R}([A])=0$, in the case
${}^2E_{6,\,2}^{16''}$.
\end{itemize}

In the case of ${}^2E_{6,\,0}^{78}$ $G$ is anisotropic; in the case of ${}^2E_{6,\,4}^{2}$ $G$ is quasi-split.
\end{thm}
\begin{proof}
By Lemma~\ref{lem:opinv} and Proposition~\ref{prop:listcomb} the Tits index of $G$ is one of the listed above.

In the case of ${}^2E_{6,\,1}^{35}$ the anisotropic kernel $G_{an}$ is a group of type ${}^2A_5$. The Cartan matrix of $E_6$ shows that
$\alpha_2=2\omega_2-\omega_4$. By Theorem~\ref{thm:levi} we have
$$
0=\beta_{G_{an}}(\alpha'_2)=-\beta_{G_{an}}(\omega_3).
$$

In the case of ${}^2E_{6,\,1}^{29}$ the anisotropic kernel $G_{an}$ is a group of type ${}^2D_4$. Denote by $O=\Spec R'$ the orbit
corresponding to $\{1,\,6\}$. We have $\alpha_1=2\omega_1-\omega_2$, so
$$
0=\beta_{{G_{an}}_O}(\alpha'_O)=-\beta_{{G_{an}}_O}(\omega_3).
$$

In the case of ${}^2E_{6,\,2}^{16'}$ the anisotropic kernel $G_{an}$ is a group of type ${}^2A_3$. We have
$\alpha_1=2\omega_1-\omega_2$, $\alpha_2=2\omega_2-\omega_4$, so
\begin{align*}
&0=\beta_{{G_{an}}_O}(\alpha'_O)=-\beta_{{G_{an}}_O}(\omega_1);\\
&0=\beta_{G_{an}}(\alpha'_4)=-\beta_{G_{an}}(\omega_2).
\end{align*}

In the case of ${}^2E_{6,\,2}^{16''}$ the anisotropic kernel $G_{an}$ is isomorphic to $R_{R'/R}(\SL_1(A))$, where $A$
is an Azumaya algebra over $R'$ with $\ind A=\deg A=3$, $O\simeq\Spec R'$. We have
$\alpha_4=2\omega_4-\omega_2-\omega_3-\omega_5$, so
$$
0=\beta_{G_{an}}(\alpha'_4)=\cores_{R'/R}(\beta_{\SL_1(A)}(\omega_1))=\cores_{R'/R}([A]).
$$
\end{proof}\medskip

\setcounter{thm}{2}\begin{thm}[$\mathbf{E_7}$]\label{thm:E7}
Every simple simply connected group $G$ of type $E_7$ over $R$ has one of the following Tits indices:

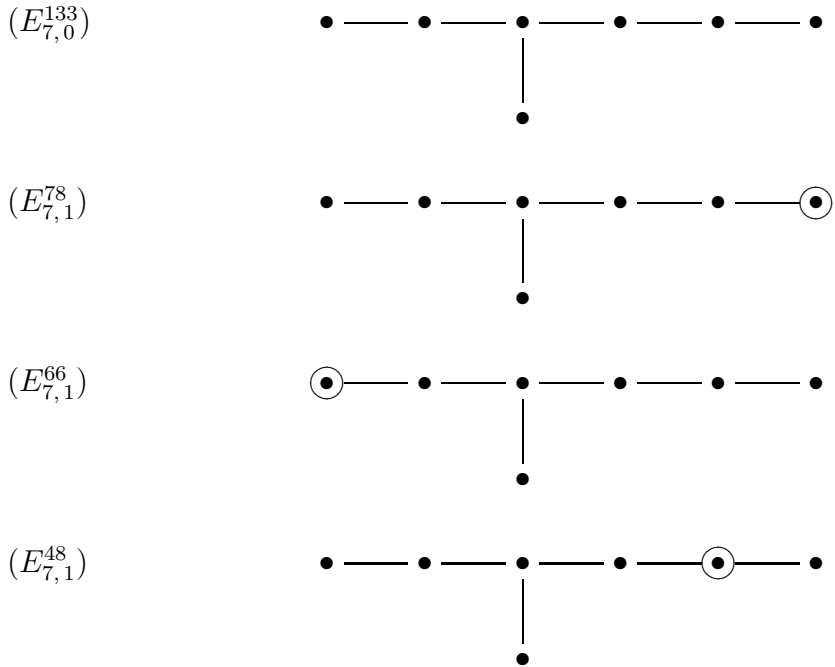
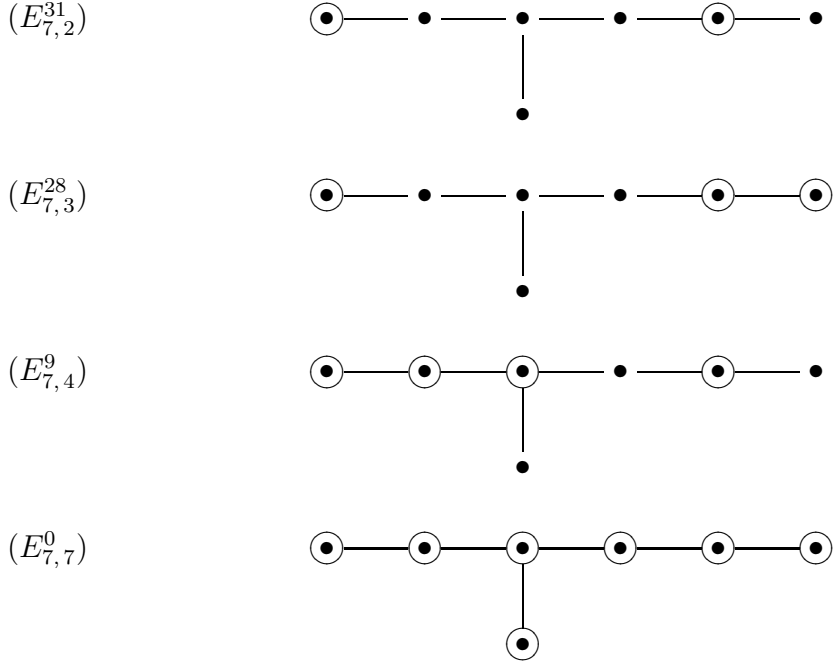
\begin{equation}\tag{$E_{7,\,0}^{133}$}
\xymatrix{
{\bullet}\ar@{-}[r]&{\bullet}\ar@{-}[r]&{\bullet}\ar@{-}[d]\ar@{-}[r]&{\bullet}\ar@{-}[r]&{\bullet}\ar@{-}[r]&{\bullet}\\
&&{\bullet}
}
\end{equation}

\begin{equation}\tag{$E_{7,\,1}^{78}$}
\xymatrix{
{\bullet}\ar@{-}[r]&{\bullet}\ar@{-}[r]&{\bullet}\ar@{-}[d]\ar@{-}[r]&{\bullet}\ar@{-}[r]&{\bullet}\ar@{-}[r]&*+[o][F]{\bullet}\\
&&{\bullet}
}
\end{equation}

\begin{equation}\tag{$E_{7,\,1}^{66}$}
\xymatrix{
*+[o][F]{\bullet}\ar@{-}[r]&{\bullet}\ar@{-}[r]&{\bullet}\ar@{-}[d]\ar@{-}[r]&{\bullet}\ar@{-}[r]&{\bullet}\ar@{-}[r]&{\bullet}\\
&&{\bullet}
}
\end{equation}

\begin{equation}\tag{$E_{7,\,1}^{48}$}
\xymatrix{
{\bullet}\ar@{-}[r]&{\bullet}\ar@{-}[r]&{\bullet}\ar@{-}[d]\ar@{-}[r]&{\bullet}\ar@{-}[r]&*+[o][F]{\bullet}\ar@{-}[r]&{\bullet}\\
&&{\bullet}
}
\end{equation}

\begin{equation}\tag{$E_{7,\,2}^{31}$}
\xymatrix{
*+[o][F]{\bullet}\ar@{-}[r]&{\bullet}\ar@{-}[r]&{\bullet}\ar@{-}[d]\ar@{-}[r]&{\bullet}\ar@{-}[r]&*+[o][F]{\bullet}\ar@{-}[r]&{\bullet}\\
&&{\bullet}
}
\end{equation}

\begin{equation}\tag{$E_{7,\,3}^{28}$}
\xymatrix{
*+[o][F]{\bullet}\ar@{-}[r]&{\bullet}\ar@{-}[r]&{\bullet}\ar@{-}[d]\ar@{-}[r]&{\bullet}\ar@{-}[r]&*+[o][F]{\bullet}\ar@{-}[r]&*+[o][F]{\bullet}\\
&&{\bullet}
}
\end{equation}

\begin{equation}\tag{$E_{7,\,4}^{9}$}
\xymatrix{
*+[o][F]{\bullet}\ar@{-}[r]&*+[o][F]{\bullet}\ar@{-}[r]&*+[o][F]{\bullet}\ar@{-}[d]\ar@{-}[r]&{\bullet}\ar@{-}[r]&*+[o][F]{\bullet}\ar@{-}[r]&{\bullet}\\
&&{\bullet}
}
\end{equation}

\begin{equation}\tag{$E_{7,\,7}^{0}$}
\xymatrix{
*+[o][F]{\bullet}\ar@{-}[r]&*+[o][F]{\bullet}\ar@{-}[r]&*+[o][F]{\bullet}\ar@{-}[d]\ar@{-}[r]&*+[o][F]{\bullet}\ar@{-}[r]&*+[o][F]{\bullet}
\ar@{-}[r]&*+[o][F]{\bullet}\\
&&*+[o][F]{\bullet}
}
\end{equation}

The possible anisotropic kernels are the following:
\begin{itemize}
\item simple simply connected anisotropic groups $H$ of type ${}^1E_6$ over $R$ with $\beta_H=0$, in the case of $E_{7,\,1}^{78}$;
\item simple simply connected anisotropic groups $H$ of type ${}^1D_6$ over $R$ with $\beta_H(\omega_5)=0$, in the case of $E_{7,\,1}^{66}$;
\item $H\times\SL_1(E)$, where $E$ is an Azumaya algebra over $R$ with $\ind E=\deg E=2$, $H$ is a simple simply connected anisotropic group of type ${}^1D_5$
over $R$ with $\beta_H(\omega_4)=[E]$, in the case of $E_{7,\,1}^{48}$;
\item $H\times\SL_1(E)$, where $E$ is an Azumaya algebra over $R$ with $\ind E=\deg E=2$, $H$ is a simple simply connected anisotropic group of type ${}^1D_4$
over $R$ with $\beta_H(\omega_1)=0$ and $\beta_H(\omega_3)=[E]$, in the case of $E_{7,\,2}^{31}$;
\item simple simply connected anisotropic groups $H$ of type ${}^1D_4$ over $R$ with $\beta_H=0$, in the case of $E_{7,\,3}^{28}$;
\item $\SL_1(A)^3$, where $A$ is an Azumaya algebra over $R$ with $\ind A=\deg A=2$, in the case of $E_{7,\,4}^{9}$.
\end{itemize}

In the case of $E_{7,\,0}^{133}$ $G$ is anisotropic; in the case of $E_{7,\,7}^{0}$ $G$ is split.
\end{thm}
\begin{proof}
By Lemma~\ref{lem:opinv} and Proposition~\ref{prop:listcomb} the Tits index of $G$ is either one of the listed above or
the following:
$$
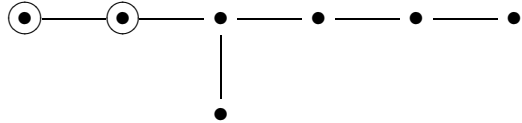
\xymatrix{
*+[o][F]{\bullet}\ar@{-}[r]&*+[o][F]{\bullet}\ar@{-}[r]&{\bullet}\ar@{-}[d]\ar@{-}[r]&{\bullet}\ar@{-}[r]&{\bullet}\ar@{-}[r]&{\bullet}\\
&&{\bullet}
}
$$

Let us first exclude the latter case. The anisotropic kernel $G_{an}$ is isomorphic to $\SL_1(A)$ for some
Azumaya algebra $A$ over $R$ with $\ind A=\deg A=6$. The Cartan matrix of $E_7$ shows that $\alpha_3=2\omega_3-\omega_1-\omega_4$.
By Theorem~\ref{thm:levi} we have
$$
0=\beta_{G_{an}}(\alpha'_3)=-\beta_{\SL_1(A)}(\omega_2)=2[A].
$$
Hence $\exp A=2$, but this contradicts  Proposition~\ref{prop:exp}.

In the case of $E_{7,\,1}^{78}$ the anisotropic kernel $G_{an}$ is of type ${}^1E_6$. We have $\alpha_7=2\omega_7-\omega_6$, so
$$
0=\beta_{G_{an}}(\alpha'_7)=-\beta_{G_{an}}(\omega_6).
$$
It follows that $\beta_{G_{an}}=0$.

In the case of $E_{7,\,1}^{66}$ the anisotropic kernel $G_{an}$ is of type ${}^1D_6$. We have $\alpha_1=2\omega_1-\omega_3$, so
$$
0=\beta_{G_{an}}(\alpha'_1)=-\beta_{G_{an}}(\omega_5).
$$

In the case of $E_{7,\,1}^{48}$ the anisotropic kernel $G_{an}$ is isomorphic to $H\times\SL_1(E)$, where $H$ is a group of
type ${}^1D_6$, $E$ is an Azumaya algebra over $R$ with $\ind E=\deg E=2$. We have $\alpha_6=2\omega_6-\omega_5-\omega_7$, so
$$
0=\beta_{G_{an}}(\alpha'_6)=-\beta_H(\omega_4)-\beta_{\SL_1(E)}(\omega_1)=-\beta_H(\omega_4)+[E].
$$

In the case of $E_{7,\,2}^{31}$ the anisotropic kernel $G_{an}$ is isomorphic to $H\times\SL_1(E)$, where $H$ is a group of
type ${}^1D_4$, $E$ is an Azumaya algebra over $R$ with $\ind E=\deg E=2$. We have $\alpha_1=2\omega_1-\omega_3$,
$\alpha_6=2\omega_6-\omega_5-\omega_7$, so
\begin{align*}
&0=\beta_{G_{an}}(\alpha'_1)=-\beta_H(\omega_1);\\
&0=\beta_{G_{an}}(\alpha'_6)=-\beta_H(\omega_3)-\beta_{\SL_1(E)}(\omega_1)=-\beta_H(\omega_3)+[E].
\end{align*}

In the case of $E_{7,\,3}^{28}$ the anisotropic kernel $G_{an}$ is of type ${}^1D_4$. We have $\alpha_1=2\omega_1-\omega_3$,
$\alpha_6=2\omega_6-\omega_5-\omega_7$, so
\begin{align*}
&0=\beta_{G_{an}}(\alpha'_1)=-\beta_{G_{an}}(\omega_1);\\
&0=\beta_{G_{an}}(\alpha'_6)=-\beta_{G_{an}}(\omega_3).
\end{align*}
It follows that $\beta_{G_{an}}=0$.

In the case of $E_{7,\,3}^{28}$ the anisotropic kernel $G_{an}$ is isomorphic to $\SL_1(A_1)\times\SL_1(A_2)\times\SL_1(A_3)$
for some Azumaya algebras $A_1$, $A_2$, $A_3$ over $R$ with $\ind A_1=\deg A_1=\ind A_2=\deg A_2=\ind A_3=\deg A_3=2$.
We have $\alpha_4=2\omega_4-\omega_2-\omega_3-\omega_5$, $\alpha_6=2\omega_6-\omega_5-\omega_7$, so
\begin{align*}
&0=\beta_{G_{an}}(\alpha'_4)=-\beta_{\SL_1(A_1)}(\omega_1)-\beta_{\SL_1(A_2)}(\omega_1)=[A_1]-[A_2];\\
&0=\beta_{G_{an}}(\alpha'_6)=-\beta_{\SL_1(A_2)}(\omega_1)-\beta_{\SL_1(A_3)}(\omega_1)=[A_2]-[A_3].
\end{align*}
By Lemma~\ref{lem:Div} $A_1\simeq A_2\simeq A_3$.
\end{proof}\medskip

\setcounter{thm}{2}\begin{thm}[$\mathbf{E_8}$]\label{thm:E8}
Every simple simply connected group $G$ of type $E_8$ over $R$ has one of the following Tits indices:

\begin{equation}\tag{$E_{8,\,0}^{248}$}
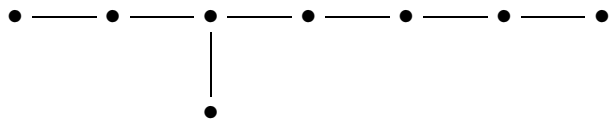
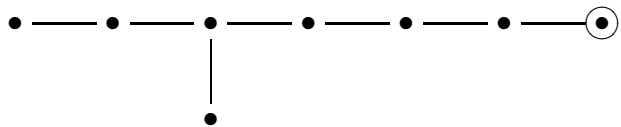
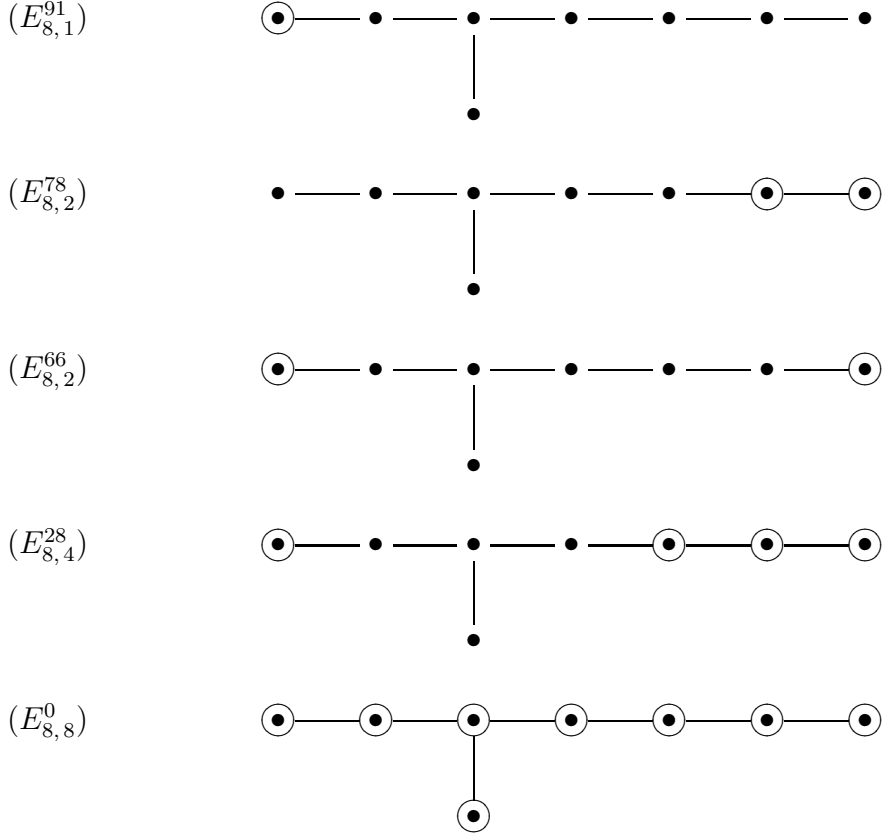
\xymatrix{
{\bullet}\ar@{-}[r]&{\bullet}\ar@{-}[r]&{\bullet}\ar@{-}[d]\ar@{-}[r]&{\bullet}\ar@{-}[r]&{\bullet}\ar@{-}[r]&{\bullet}\ar@{-}[r]&{\bullet}\\
&&{\bullet}
}
\end{equation}

\begin{equation}\tag{$E_{8,\,1}^{133}$}
\xymatrix{
{\bullet}\ar@{-}[r]&{\bullet}\ar@{-}[r]&{\bullet}\ar@{-}[d]\ar@{-}[r]&{\bullet}\ar@{-}[r]&{\bullet}\ar@{-}[r]&{\bullet}\ar@{-}[r]&*+[o][F]{\bullet}\\
&&{\bullet}
}
\end{equation}

\begin{equation}\tag{$E_{8,\,1}^{91}$}
\xymatrix{
*+[o][F]{\bullet}\ar@{-}[r]&{\bullet}\ar@{-}[r]&{\bullet}\ar@{-}[d]\ar@{-}[r]&{\bullet}\ar@{-}[r]&{\bullet}\ar@{-}[r]&{\bullet}\ar@{-}[r]&{\bullet}\\
&&{\bullet}
}
\end{equation}

\begin{equation}\tag{$E_{8,\,2}^{78}$}
\xymatrix{
{\bullet}\ar@{-}[r]&{\bullet}\ar@{-}[r]&{\bullet}\ar@{-}[d]\ar@{-}[r]&{\bullet}\ar@{-}[r]&{\bullet}\ar@{-}[r]&*+[o][F]{\bullet}\ar@{-}[r]&*+[o][F]{\bullet}\\
&&{\bullet}
}
\end{equation}

\begin{equation}\tag{$E_{8,\,2}^{66}$}
\xymatrix{
*+[o][F]{\bullet}\ar@{-}[r]&{\bullet}\ar@{-}[r]&{\bullet}\ar@{-}[d]\ar@{-}[r]&{\bullet}\ar@{-}[r]&{\bullet}\ar@{-}[r]&{\bullet}\ar@{-}[r]&*+[o][F]{\bullet}\\
&&{\bullet}
}
\end{equation}

\begin{equation}\tag{$E_{8,\,4}^{28}$}
\xymatrix{
*+[o][F]{\bullet}\ar@{-}[r]&{\bullet}\ar@{-}[r]&{\bullet}\ar@{-}[d]\ar@{-}[r]&{\bullet}\ar@{-}[r]&*+[o][F]{\bullet}\ar@{-}[r]&*+[o][F]{\bullet}\ar@{-}[r]&*+[o][F]{\bullet}\\
&&{\bullet}
}
\end{equation}

\begin{equation}\tag{$E_{8,\,8}^{0}$}
\xymatrix{
*+[o][F]{\bullet}\ar@{-}[r]&*+[o][F]{\bullet}\ar@{-}[r]&*+[o][F]{\bullet}\ar@{-}[d]\ar@{-}[r]&*+[o][F]{\bullet}\ar@{-}[r]&*+[o][F]{\bullet}\ar@{-}[r]
&*+[o][F]{\bullet}\ar@{-}[r]&*+[o][F]{\bullet}\\
&&*+[o][F]{\bullet}
}
\end{equation}

The possible anisotropic kernels are the following:
\begin{itemize}
\item simple simply connected anisotropic groups $H$ of type $E_7$ over $R$ with $\beta_H=0$, in the case of $E_{8,\,1}^{133}$;
\item simple simply connected anisotropic groups $H$ of type ${}^1D_7$ over $R$ with $\beta_H=0$, in the case of $E_{8,\,1}^{91}$;
\item simple simply connected anisotropic groups $H$ of type ${}^1E_6$ over $R$ with $\beta_H=0$, in the case of $E_{8,\,2}^{78}$;
\item simple simply connected anisotropic groups $H$ of type ${}^1D_6$ over $R$ with $\beta_H=0$, in the case of $E_{8,\,2}^{66}$;
\item simple simply connected anisotropic groups $H$ of type ${}^1D_4$ over $R$ with $\beta_H=0$, in the case of $E_{8,\,4}^{28}$.
\end{itemize}

In the case of $E_{8,\,0}^{248}$ $G$ is anisotropic; in the case of $E_{8,\,8}^{0}$ $G$ is split.
\end{thm}
\begin{proof}
By Lemma~\ref{lem:opinv} and Proposition~\ref{prop:listcomb} the Tits index of $G$ is one of the listed above.

In the case of $E_{8,\,1}^{133}$ the anisotropic kernel $G_{an}$ is of type $E_7$. The Cartan matrix of $E_8$ shows that
$\alpha_8=2\omega_8-\omega_7$. By Theorem~\ref{thm:levi} we have
$$
0=\beta_{G_{an}}(\alpha'_8)=-\beta_{G_{an}}(\omega_7).
$$
It follows that $\beta_{G_{an}}=0$.

In the case of $E_{8,\,1}^{91}$ the anisotropic kernel $G_{an}$ is of type ${}^1D_7$. We have $\alpha_1=2\omega_1-\omega_3$, so
$$
0=\beta_{G_{an}}(\alpha'_1)=-\beta_{G_{an}}(\omega_6).
$$
It follows that $\beta_{G_{an}}=0$.

In the case of $E_{8,\,2}^{78}$ the anisotropic kernel $G_{an}$ is of type ${}^1E_6$. We have $\alpha_7=2\omega_7-\omega_6-\omega_8$, so
$$
0=\beta_{G_{an}}(\alpha'_7)=-\beta_{G_{an}}(\omega_6).
$$
It follows that $\beta_{G_{an}}=0$.

In the case of $E_{8,\,2}^{66}$ the anisotropic kernel $G_{an}$ is of type ${}^1D_6$. We have $\alpha_1=2\omega_1-\omega_3$,
$\alpha_8=2\omega_8-\omega_7$, so
\begin{align*}
&0=\beta_{G_{an}}(\alpha'_1)=-\beta_{G_{an}}(\omega_5);\\
&0=\beta_{G_{an}}(\alpha'_8)=-\beta_{G_{an}}(\omega_1).
\end{align*}
It follows that $\beta_{G_{an}}=0$.

In the case of $E_{8,\,4}^{28}$ the anisotropic kernel $G_{an}$ is of type ${}^1D_4$. We have $\alpha_1=2\omega_1-\omega_3$,
$\alpha_6=2\omega_6-\omega_5-\omega_7$, so
\begin{align*}
&0=\beta_{G_{an}}(\alpha'_1)=-\beta_{G_{an}}(\omega_1);\\
&0=\beta_{G_{an}}(\alpha'_6)=-\beta_{G_{an}}(\omega_3).
\end{align*}
It follows that $\beta_{G_{an}}=0$.
\end{proof}\medskip

\setcounter{thm}{2}\begin{thm}[$\mathbf{F_4}$]\label{thm:F4}
Every simple simply connected group $G$ of type $F_4$ over $R$ has one of the following Tits indices:

\begin{equation}\tag{$F_{4,\,0}^{52}$}
\xymatrix{
{\bullet}\ar@{-}[r]&{\bullet}\ar@{=>}[r]&{\bullet}\ar@{-}[r]&{\bullet}
}
\end{equation}

\begin{equation}\tag{$F_{4,\,1}^{21}$}
\xymatrix{
{\bullet}\ar@{-}[r]&{\bullet}\ar@{=>}[r]&{\bullet}\ar@{-}[r]&*+[o][F]{\bullet}
}
\end{equation}

\begin{equation}\tag{$F_{4,\,4}^{0}$}
\xymatrix{
*+[o][F]{\bullet}\ar@{-}[r]&*+[o][F]{\bullet}\ar@{=>}[r]&*+[o][F]{\bullet}\ar@{-}[r]&*+[o][F]{\bullet}
}
\end{equation}

The possible anisotropic kernels in the case of $F_{4,\,1}^{21}$ are simple simply connected anisotropic groups $H$ of type $B_3$ over $R$ with $\beta_H=0$.

In the case of $F_{4,\,0}^{52}$ $G$ is anisotropic; in the case of $F_{4,\,4}^{0}$ $G$ is split.
\end{thm}
\begin{proof}
By Lemma~\ref{lem:opinv} and Proposition~\ref{prop:listcomb} the Tits index of $G$ is either one of the listed above or
one of the following:
$$
\xymatrix{
*+[o][F]{\bullet}\ar@{-}[r]&{\bullet}\ar@{=>}[r]&{\bullet}\ar@{-}[r]&{\bullet}
}
$$

$$
\xymatrix{
*+[o][F]{\bullet}\ar@{-}[r]&{\bullet}\ar@{=>}[r]&{\bullet}\ar@{-}[r]&*+[o][F]{\bullet}
}
$$

Let us exclude the two latter cases. In the first of them the anisotropic kernel $G_{an}$ is of type $C_3$.
The Cartan matrix of $F_4$ shows that $\alpha_1=2\omega_1-\omega_2$. By Theorem~\ref{thm:levi} we have
$$
0=\beta_{G_{an}}(\alpha'_1)=-\beta_{G_{an}}(\omega_3).
$$
It follows that $\beta_{G_{an}}=0$, in contradiction with Proposition~\ref{prop:sp}.

In the second case the anisotropic kernel $G_{an}$ is of type $C_2$. We have $\alpha_4=2\omega_4-\omega_3$, so
$$
0=\beta_{G_{an}}(\alpha'_4)=-\beta_{G_{an}}(\omega_1).
$$
It follows that $\beta_{G_{an}}=0$, in contradiction with Proposition~\ref{prop:sp}.

In the case of $F_{4,\,1}^{21}$ $G_{an}$ is of type $B_3$. We have $\alpha_4=2\omega_4-\omega_3$, so
$$
0=\beta_{G_{an}}(\alpha'_4)=-\beta_{G_{an}}(\omega_3).
$$
It follows that $\beta_{G_{an}}=0$.
\end{proof}\medskip

\setcounter{thm}{2}\begin{thm}[$\mathbf{G_2}$]\label{thm:G2}
Every simple simply connected group $G$ of type $G_2$ over $R$ has one of the following Tits indices:

\begin{equation}\tag{$G_{2,\,0}^{14}$}
\xymatrix{
{\bullet}\ar@3{<-}[r]&{\bullet}
}
\end{equation}

\begin{equation}\tag{$G_{2,\,2}^{0}$}
\xymatrix{
*+[o][F]{\bullet}\ar@3{<-}[r]&*+[o][F]{\bullet}
}
\end{equation}

In the case of $G_{2,\,0}^{14}$ $G$ is anisotropic; in the case of $G_{2,\,2}^{0}$ $G$ is split.
\end{thm}
\begin{proof}
By Lemma~\ref{lem:opinv} and Proposition~\ref{prop:listcomb} the Tits index of $G$ is either one of the listed above or
the following:
$$
\xymatrix{
{\bullet}\ar@3{<-}[r]&*+[o][F]{\bullet}
}
$$
We need to exclude the latter case. The anisotropic kernel $G_{an}$ is isomorphic to $\SL_1(A)$ for some Azumaya
algebra $A$ over $R$ with $\ind A=\deg A=2$. The Cartan matrix of $G_2$ shows that $\alpha_2=2\omega_2-3\omega_1$.
By Theorem~\ref{thm:levi} we have
$$
0=\beta_{G_{an}}(\alpha'_2)=-3\beta_{\SL_1(A)}(\omega_1)=-3[A].
$$
But by Proposition~\ref{prop:exp} $2[A]=0$, hence $[A]=0$, a contradiction.
\end{proof}\medskip

\section{Existence of indices}\label{sec:exist}

For the sake of completeness we give here a new uniform proof of the existence
of indices of exceptional inner type over fields (note that all indices
of outer types ${}^2E_6$, ${}^3D_4$, and ${}^6D_4$ appear already over
number fields).

\begin{thm}\label{th:ex} For any field $F$ and any prescribed Tits index of exceptional inner type
listed in Section~\ref{sec:ind}, there exists a field extension $E/F$ and a simple
algebraic group $G$ over $E$ having that Tits index.
\end{thm}
\begin{proof}
Denote by $H_0$ the derived subgroup of
the standard Levi subgroup of a parabolic subgroup $P_0$ in the split adjoint
group $G_0^{ad}$ over $F$. Now consider a \emph{generic torsor} $\zeta$ under
$H_0$ over an extension $E/F$. Recall that to construct $\zeta$ one chooses a
faithful representation $H_0\to\GL_n$, considers $E=F(\GL_n/H_0)$, and then
takes the image in $\HH^1(E,\,H_0)$ of the generic point in $\GL_n/H_0(E)$
under the connecting map  arising from the sequence
$$
1\to H_0\to\GL_n\to\GL_n/H_0\to 1.
$$
After that we take the image $\xi$ of $\zeta$ in $\HH^1(E,\,G_0^{ad})$ and
consider the corresponding group $G$ over $E$. Obviously $G$ has a parabolic
subgroup $P$ whose Levi part is isomorphic to the group $H$ corresponding to
$\zeta$. In general it may happen that $H$ is isotropic, that is the Tits
index of $G$ contains more circled vertices than desired. Our goal is to show
that if $P_0$ corresponds to one of the indices listed in Section~\ref{sec:ind},
this is not the case.

To this end we employ an invariant $\cd_p(X)$ of a projective homogeneous
variety $X$ called the \emph {$p$-relative canonical dimension} of $X$;
see \cite{KM} for the definition and basic properties. We take $X$ to be
the variety of Borel subgroups of $G$. It is shown in
\cite[Proposition~6.1]{PSZ} that $\cd_p(X)$ depends in an explicit monotonic way
on a certain discrete invariant $J_p(G)$ of $G$ (the~\emph{$J$-invariant}). By
\cite[Corollary~5.19]{PSZ} this invariant is the same for the group $G$
itself and the derived subgroup $H'$ of any parabolic subgroup of $G$.
Also, if a group $H$ corresponds to a generic torsor, $J_p(G)$ takes
the maximal possible value, which is computed in \cite[Example~4.7]{PSZ}.

Now assume that the anisotropic kernel $H'$ of $G$
is less than $H$. From the one hand side, $\cd_p(X)$ can be computed in terms
of $J_p(H)$, which is known, since $H$ is generic. From the other hand side, it
can be computed in terms of $J_p(H')$, which does not exceed the known maximal
possible value. If these values are distinct, we get a contradiction. Looking
at the following table we see that for any two indices of the same type one can
find $p$ such that the maximal possible values of $\cd_p(X)$ differ, and we are
done.

\medskip

{\renewcommand{\arraystretch}{1.3}

\begin{tabular}{l|lcl|l}
Index&Maximal value of $\cd_p(X)$&&Index&Maximal value of $\cd_p(X)$\\
\cline{1-2}\cline{4-5}
${}^1E_{6,\,0}^{78}$&$3$, $p=2$; $16$, $p=3$&             &$E_{8,\,0}^{248}$&$60$, $p=2$; $28$, $p=3$; $24$, $p=5$\\
${}^1E_{6,\,2}^{28}$&$3$, $p=2$&                          &$E_{8,\,1}^{133}$&$17$, $p=2$; $8$, $p=3$\\
${}^1E_{6,\,2}^{16}$&$2$, $p=3$&                          &$E_{8,\,1}^{91}$&$14$, $p=2$\\
${}^1E_{6,\,6}^{0}$&$0$&                                  &$E_{8,\,2}^{78}$&$3$, $p=2$; $8$, $p=3$\\
\cline{1-2}
$E_{7,\,0}^{133}$&$18$, $p=2$; $8$, $p=3$&                &$E_{8,\,2}^{66}$&$8$, $p=2$\\
$E_{7,\,1}^{78}$&$3$, $p=2$; $8$, $p=3$&                  &$E_{8,\,4}^{28}$&$3$, $p=2$\\
$E_{7,\,1}^{66}$&$9$, $p=2$&                              &$E_{8,\,8}^{0}$&$0$\\
\cline{4-5}
$E_{7,\,1}^{48}$&$10$, $p=2$&                             &$F_{4,\,0}^{52}$&$3$, $p=2$; $8$, $p=3$\\
$E_{7,\,2}^{31}$&$6$, $p=2$&                              &$F_{4,\,1}^{21}$&$3$, $p=2$\\
$E_{7,\,3}^{28}$&$3$, $p=2$&                              &$F_{4,\,4}^{0}$&$0$\\
\cline{4-5}
$E_{7,\,4}^{9}$&$1$, $p=2$&                               &$G_{2,\,0}^{14}$&$3$, $p=2$\\
$E_{7,\,7}^{0}$&$0$&                                      &$G_{2,\,2}^{0}$&$0$
\end{tabular}
}



\end{proof}

\end{document}